\documentclass[12pt,reqno,a4paper]{amsart}

\usepackage[pdftex]{graphicx}
\usepackage{subcaption}
\usepackage{pgfplots}
\usepackage{tikz}  
\usetikzlibrary{positioning, arrows.meta, shadows, calc, decorations.pathreplacing,matrix}
\usepackage{amsmath,amssymb,amsthm,mathtools, amsfonts, hyperref}

\pgfplotsset{compat=1.7}

\usepackage{a4wide}
\usepackage{amsmath}
\usepackage{amssymb,amsthm,hyperref,caption}
\usepackage{xcolor}
\definecolor{meinBlau}{rgb}{0.2,0.2,0.9} 
\definecolor{blau}{rgb}{0,0,0.75} 
\definecolor{rot}{rgb}{0.74,0,0} 
\definecolor{darkgreen}{rgb}{0,0.4,0.1} 
\definecolor{codegreen}{rgb}{0,0.6,0} 
\definecolor{codeblue}{rgb}{0,0.1,0.7} 

\definecolor{myblauA}{rgb}{0,0.1,1} 
\definecolor{myblauB}{rgb}{0,0.2,0.8} 
\definecolor{myblauC}{rgb}{0,0.3,0.7}
\definecolor{myblauD}{rgb}{0,0.4,0.55}
\definecolor{myblauE}{rgb}{0,0.5,0.4}
\definecolor{myblauF}{rgb}{0,0.6,0.3}

\hypersetup{colorlinks,linkcolor=blau,citecolor=blue,urlcolor=meinBlau}
\usepackage{times}
\usepackage{enumitem}
\usepackage{verbatim}
\usepackage{listings}
\usepackage{fancyvrb}

\lstdefinestyle{Rstyle}{
language=R, 
backgroundcolor=\color{white}, 
commentstyle=\color{codegreen}, 
keywordstyle=\color{codeblue}, 
numberstyle=\color{codeblue}, 
stringstyle=\color{codegreen}, 
basicstyle=\ttfamily\tiny, 
morekeywords={TRUE,FALSE}, 
deletekeywords={data,frame,length,as,character}, 
showstringspaces=false 
}

\allowdisplaybreaks

\newtheorem{theorem}{Theorem}
\newtheorem{lem}[theorem]{Lemma}

\newtheorem{corollary}{Corollary}
\newtheorem{prop}{Proposition}
\newtheorem{proposition}{Proposition}

\theoremstyle{definition}

\newtheorem{remark}{Remark}
\newtheorem{example}{Example}

\newtheorem{definition}{Definition}

\def\P{{\mathbb {P}}}
\def\E{{\mathbb {E}}}
\def\Var{{\mathbb {V}}}

\newcommand{\up}{\uparrow}
\newcommand{\down}{\downarrow}

\newcommand{\R}{\ensuremath{\mathbb{R}}}

\DeclareMathOperator{\fc}{FC}

\DeclareMathOperator{\law}{\overset{\mathcal{L}}{=}}

\DeclareMathOperator{\maxP}{maxP}

\newcommand{\fv}{\ensuremath{\mathfrak{v}}}
\newcommand{\round}[1]{\left(#1\right)}

\begin{document}

\author[A.~Clay]{Alexander Clay}
\address{Alexander Clay \\ Department of Mathematics\\
University of Southern California\\
Los Angeles, CA, USA\\} %
\email{ajclay@usc.edu}

\author[M.~Kuba]{Markus Kuba}
\address{Markus Kuba\\
Department Applied Mathematics and Physics\\
University of Applied Sciences - Technikum Wien\\
H\"ochst\"adtplatz 5, 1200 Wien} %
\email{kuba@technikum-wien.at}

\author[R.~Tripathi]{Raghavendra Tripathi}
\address{Raghavendra Tripathi\\
Department of Mathematics, Division of Science, New York University, Abu Dhabi, UAE}
\email{r.tripathi@nyu.edu}

\title[Card Guessing after a Shelf Shuffle]{On card guessing after an asymmetric single-shelf shuffle}


\begin{abstract} 
We provide a definitive analysis of the number of correct guesses in the complete-feedback card guessing game after an asymmetric single-shelf shuffle with parameter $p\in (0, 1)$. We explicitly describe the optimal strategy that maximizes the expected number of correct guesses. We study the number of correct guesses, under an optimal strategy, using methods from analytic combinatorics. In addition, we find the explicit distribution for the number of correct guesses, and thus find the mean and the variance for the number of correct guesses. We show that the distribution of the number of correct guesses is log-concave. We also study the limiting behaviour of the number of correct guesses as the number of cards goes to infinity. In particular, we prove a (local) central limit theorem and a large deviation principle with an explicit rate function. We also prove phase transitions for the number of correct guesses near $p=0$ and $p=1$. Prior to this work, only the optimal strategy and the expected number of correct guesses under the optimal strategy were known for the special case of $p=1/2$.  
\end{abstract}

\maketitle
\section{Introduction and Main Results}
\label{sec:Intro}
The analysis of card shuffling has a long and rich history, dating back to the work of Markov \cite{Markov} and Poincar\'e~\cite{Poincare}. Card shuffling intersects with several branches of mathematics, such as representation theory, combinatorics, probability theory, and statistics. We refer the reader to the recent book~\cite{DiaconisFulman2023Shuffling} by Diaconis and Fulman for a modern and comprehensive account. 

Famous shuffling techniques include the riffle shuffle~\cite{BayerDiaconis,DiaconisGraham1981,Gilbert1955}, also called dove-tail shuffling, as well as the shelf shuffle~\cite{DiaconisFulmanHolmes2013}; we refer the reader to~\cite{Sellke, Kraitchik1942Recreations} for references concerning other shuffling techniques. One of the first questions in the analysis of card shuffling is to quantify how fast the shuffling scheme mixes the cards~\cite{AldousDiaconis1986}. This is often done by proving mixing time estimates and the cutoff phenomenon~\cite{Sellke,diaconis1996cutoff}. A related aspect is to study the statistics (e.g., descents, inversions, fixed points, etc.) of the permutations induced by a shuffle. In many practical situations, one is often interested in a particular statistic, and it suffices for that statistic to become sufficiently random. 

Another fascinating question is to study card guessing games~\cite{BayerDiaconis} with different types of feedback. In card-guessing games, a player sequentially guesses a shuffled deck of cards and is possibly provided some feedback (e.g., if the guess is correct or not) after each guess. The key quantity of interest here is the number of correct guesses the player can make. Obviously, the number of correct guesses depends on the guessing strategy of the player; for example, a player may decide to guess the card $1$ at each step, in which case the player makes exactly one correct guess almost surely. The goal of the player is to maximize the expected number of correct guesses over all possible guessing strategies. Such a strategy is called an \emph{optimal strategy}. Optimal strategies are not necessarily unique, and they depend on the shuffling scheme. Finding an optimal strategy is also a non-trivial problem in most cases. Even when an optimal strategy is known, it may be difficult to implement in practice, and in such cases, there has also been significant interest in investigating strategies that are close to optimal but are more practical. The number of correct guesses under these strategies has been investigated by several authors~\cite{OttoliniSteiner2023, HeOttolini2021, OT2024, Diaconis2022a, DiaconisGraham1981, KP2024}. Card guessing games can provide a qualitative understanding of the effectiveness of a shuffle, and they also have important applications in clinical trials and fraud detection~\cite{BlackwellHodges1957, Efron1971, Diaconis1978, Fisher1936, Proschan91}. Throughout this paper, whenever we mention the number of correct guesses, it is assumed that the player uses an optimal strategy.

In this paper, we study the asymmetric shelf shuffle (with a single shelf), which is a generalization of the shelf shuffle introduced in~\cite{DiaconisFulmanHolmes2013}. We investigate the distributional properties of the number of correct guesses in the complete feedback game (under an optimal strategy). We begin with the definition. 
\begin{definition}[Asymmetric single shelf shuffle]
We start with an ordered deck of $n$ cards labelled $1$ through $n$ and a parameter $p\in(0,1)$. Cards are drawn from the bottom of the deck and placed in a pile. 
When a card is drawn, it is placed at the top with probability $p$ and at the bottom of the pile with probability $1-p$.
\end{definition}
Throughout this paper, we will only consider the single-shelf shuffle. One can immediately generalize this to an asymmetric shelf shuffle with multiple shelves. We refer the reader to~\cite{tri26analysis, Tri26cut} for more details on the asymmetric shelf shuffle. Diaconis, Fulman, and Holmes~\cite{DiaconisFulmanHolmes2013} initiated the study of the shelf-shuffle when $p=1/2$. For concreteness, we refer to this as the \emph{symmetric shelf shuffle}. While~\cite{DiaconisFulmanHolmes2013} proved several interesting results about shelf-shuffles, many natural questions remain still unexplored. Very recently,~\cite{ChenOttolini2025} established the cutoff phenomenon for the symmetric shelf-shuffle and determined the exact cutoff profile, and the first author~\cite{Clay2025b} studied descents and inversions. Card guessing after a symmetric shelf shuffle with a single shelf was analyzed in~\cite{Clay2025a, Tripathi2026}. For the asymmetric shelf shuffle, the law was determined recently, and cutoff was proven in~\cite{Tri26cut}. Generating functions for several statistics of the shelf shuffle were obtained in~\cite{tri26analysis}. 

\begin{example}[Single-shelf shuffle with $20$ cards with $p=0.7$]
Except for the card $20$, we can imagine that every card is assigned an independent up ($\up$) or down ($\down$) label with probability $p$ and $1-p$, respectively. Equivalently, for $1\leq j\leq n-1$, let $\varepsilon_j=1$ if card $j$ is placed at the top of the pile, and let $\varepsilon_j=0$ if it is placed at the bottom; then the $\varepsilon_j$ are independent Bernoulli$(p)$ random variables. All the cards assigned the label $\up$ are placed above the card $n$ in increasing order, while all the cards labelled $\down$ are placed below the card $n$ in decreasing order to obtain the shuffled deck. We illustrate this with an example.
\begin{center}
\begin{tikzpicture}[x=0.6cm,y=0.6cm]
\foreach \i in {1,...,20} {
  \draw (\i-1,1) rectangle (\i,2);
  \node at (\i-0.5,1.5) {\i};
}
\foreach \i/\s in {1/$\down$, 2/$\up$,3/$\down$,4/$\up$,5/$\up$,6/$\down$,7/$\up$,8/$\down$,9/$\up$,10/$\down$,11/$\down$,12/$\up$,13/$\up$,14/$\down$,15/$\up$,16/$\down$,17/$\down$,18/$\up$,19/$\down$,20/} {
  \draw (\i-1,0) rectangle (\i,1);
  \node at (\i-0.5,0.5) {\s};
}
\end{tikzpicture}
\end{center}
This leads to the following shuffled deck:
\begin{center}
\begin{tikzpicture}[x=0.6cm,y=0.6cm]
\foreach \i [count=\n from 1] in {2,4, 5, 7, 9, 12,13,15,18, 20,19,17,16,14,11,10,8, 6,3,1} {
  \draw (\n-1,0) rectangle (\n,1);
  \node at (\n-0.5,0.5) {\i};
}
\end{tikzpicture}
\end{center}

Note that for any $p$, a single-shelf asymmetric shelf always produces a \emph{unimodal permutation}. That is, in the shuffled deck, the cards are increasing first until the card $n$, and thereafter the remaining cards are in decreasing order. 
\end{example}
Since the card $1$ is dealt last, it goes to the top in the final deck with probability $p$ and to the bottom with probability $1-p$. In fact, it is straightforward to compute the position matrix of a single-shelf asymmetric shelf shuffle, which generalizes \cite[Theorem 1.1]{Clay2025a}.
\begin{prop}[Position matrix of an asymmetric shelf shuffle]
\label{prop:posi}
Suppose a deck of $n$ cards is shuffled once in an asymmetric shelf-shuffle with a single shelf and parameter $p\in (0, 1)$.
Let $m_{i,j}$ be the probability that card $i$ lands in position $j$.
Then, we have
\[
m_{i,j}
=
p^{j}(1-p)^{i-j}\binom{i-1}{j-1} + p^{i-1-(n-j)}(1-p)^{n-j+1}\binom{i-1}{n-j},
\]
for $1 \le i \le n$ and $1 \le j \le n$. In particular, for $i=n$ we have 
\[
m_{n,j}= p^{j-1}(1-p)^{n-j}\binom{n-1}{j-1},\quad 1\le j\le n.
\]
\end{prop}
Here and throughout, we use the convention that $\binom{0}{0} = 1$ and $\binom{a}{b}=0$ if $a<b$. It is easy to check that $m$ is a doubly stochastic matrix and $m_{i,n-j+1} = m_{i,j}$ for all $i,j$. 
\begin{proof}
Immediately before selecting the card labelled $i$, we have exactly $n-i$ cards in the pile. The $i$th card is placed at the bottom with probability $1-p$
or the top of the pile with probability $p$. If the card is placed at the top, we select $j-1$ cards from the remaining $i-1$ cards
and place them at the top, leading to a common factor $p^{j}(1-p)^{i-j}$ and the number of ways to choose the $j-1$ cards. If the card is placed at the bottom, we may choose $n-j$ cards from the remaining $i - 1$ cards and place them at the bottom. This gives a common factor $p^{i-1-(n-j)}(1-p)^{n-j+1}$, which takes into account that card $i$ went to the bottom, as well as $n-j$ cards too, and the remaining cards to the top. As these events are disjoint, the result readily follows.
\end{proof} 
For concreteness, we formally describe the complete feedback game. Throughout, a shuffled deck will mean a deck of cards shuffled once by an asymmetric single-shelf shuffle with parameter $p$. 
\begin{definition}[Complete feedback game]
Consider a deck of shuffled cards. A player guesses the cards sequentially from the top of the deck. After each guess, the player is shown the card, and the card is removed from the deck. The objective of the player is to maximize the expected number of correct guesses, and a strategy that maximizes the expected number of correct guesses is called an \emph{optimal strategy}. Unless stated otherwise, we will always assume that cards are guessed by an optimal guessing strategy. 
\end{definition}

Based on extensive numerical simulations,~\cite{DiaconisFulmanHolmes2013} formulated several conjectures about the optimal strategy and the asymptotics of the number of correct guesses for the symmetric shelf shuffle. Rigorous investigation in this direction was recently initiated by the first author ~\cite{Clay2025a}, who proved the conjectured optimal strategy for complete feedback games and further showed that under the optimal strategy, the expected number of correct guesses is $3n/4$. 

We now describe the optimal strategy in a complete feedback game after a \emph{symmetric shelf shuffle} from \cite[Theorem 1.4]{Clay2025a}. Let $T$ be the first time the card $n$ or $n-1$ is revealed; in other words, $T$ is the minimum between the location of card $n$ and $n-1$ in the shuffled deck. It is easy to see that after time $T$, the player should guess the remaining cards in decreasing order. On the first step, guessing card $1$ maximizes the probability of the guess being correct. For $2\leq t\leq T$, the player should guess the card with label $X_{t-1}+1$. That is, the player should guess the card immediately following the previously shown cards until one sees the card with label $n$ or $n-1$. 
\begin{example}[Optimal strategy for the symmetric shelf shuffle]
Take $n = 20$ and suppose the shuffled deck is: 

\begin{center}
\begin{tikzpicture}[x=0.6cm,y=0.6cm]
\foreach \i [count=\n from 1] in {1,2,5,10,11,19,20,18,17,16, 15,14,13,12,9,8,7,6,4,3} {
  \draw (\n-1,0) rectangle (\n,1);
  \node at (\n-0.5,0.5) {\i};
}
\end{tikzpicture}
\end{center}
The player guesses the card $1$ at the first position (and gets it correct). Then the player guesses $2$ and $3$ on the next two cards, respectively. Note that the third guess is incorrect. When this card is revealed to be $5$, on the next card the player should guess $6$. Continuing this way, the player guesses $1, 2, 3, 6, 11, 12$. At this step, the revealed card is $19$. From here on, the player guesses the remaining cards in decreasing order. Thus, the full sequence of the card guesses in this case is: 
\begin{center}
\begin{tikzpicture}[x=0.6cm,y=0.6cm]
\foreach \i [count=\n from 1] in {1,2,3,6,11,12,20,18,17,16, 15,14,13,12,9,8,7,6,4,3} {
  \draw (\n-1,0) rectangle (\n,1);
  \node at (\n-0.5,0.5) {\i};
}
\end{tikzpicture}
\end{center}
One can calculate by matching the two diagrams that the player makes $3+14=17$ correct guesses. The last $14$ cards (after the card $19$ was revealed) are guaranteed correct guesses.  
\end{example}

For $\frac{1}{2}\le p<1$, the optimal strategy stays the same as in the case of $p=1/2$. If $p<\frac{1}{2}$, the structure is asymptotically similar, but the optimal strategy changes. It is evident that when $p<1/2$, the card $1$ is not necessarily the most likely card to appear on top of the shuffled deck. In particular, we do not always guess card $1$ on the first round. In fact, it follows from Proposition~\ref{prop:posi} (see Lemma~\ref{lem:firstcard}) that the most likely card in position $1$ is card $1$ or card $n$ with probability $p$ and $(1-p)^{n-1}$, respectively. Therefore, the solution to the equation $p=(1-p)^{n-1}$ allows us to distinguish which label appears with the highest probability on the next step. This also suggests that the first guess (and therefore optimal strategy) need not be unique. 

We now explain the optimal guessing strategy. Let $\fc_n$ denote the first card in the shuffled deck. At the first step, one should guess the label $j$ which maximizes the probability $\mathbb{P}(\fc_n=j)$. Suppose that the first card is revealed and $\fc_n=t$. Note that this means that the cards $1,\ldots, t-1$ appear at the bottom of the shuffled deck in decreasing order. These $(t-1)$ cards will eventually be guessed correctly with probability $1$ between locations $n-t+2$ and $n$. Furthermore, on the event $\fc_n=t$, the cards with labels $t+1, t+2, \ldots, n$ appear between positions $2$ and $n-t+1$ and (up to relabelling these cards) have the same distribution as an asymmetric single-shelf shuffle with $(n-t)$ cards. We can proceed inductively, and this leads to the following optimal strategy for guessing in the complete feedback game.

\begin{proposition}[Optimal strategy for asymmetric shelf shuffle]
\label{prop:Strategy}
For $p\in[\frac12,1]$, we proceed as in the case $p=\frac12$: we guess the card labelled $1$ as the first guess and then always guess the card immediately following the previously shown card in all other positions until the card $n-1$ or $n$ is shown. After that, the player should guess the remaining cards in descending order, and is guaranteed to get these correct.

For $p\in (0,\frac{1}{2})$, let $\nu=\lfloor \ln(p)/\ln(1-p)\rfloor+1$. We guess the label 
\[
\maxP_n= \arg\max_{j}\mathbb{P}(\fc_n=j) = 
\begin{cases}
n,\quad 2\le n\le \nu,\\
1,\quad n>\nu.
\end{cases}
\]
as the first card, breaking ties by choosing $n$. After the card, say labelled $j$, is shown to the guesser, we proceed similarly, but with the reduced set of indices $\{j+1,\dots,n\}$ and guess the label from $\{j+1,n\}$ corresponding to the maximum of $\{p,(1-p)^{n-j-1}\}$ until one sees the card $n$ or $n-1$. After that, the player should guess all the remaining cards in decreasing order.
\end{proposition}

\begin{example}[Optimal strategy for $p=0.20$ when $n<\nu$]
For $p=0.20$, we have $q=0.80$ and
\[
\nu=\left\lfloor\frac{\log(0.20)}{\log(0.80)}\right\rfloor+1=8.
\]
Take $n=6<\nu$, and suppose that the shelf word is
\[
(\varepsilon_1,\ldots,\varepsilon_5)=(0,0,1,0,0).
\]
Thus, the shuffled deck, shown from top to bottom, is
\begin{center}
\begin{tikzpicture}[x=0.6cm,y=0.6cm]
\foreach \i [count=\n from 1] in {3,6,5,4,2,1} {
  \draw (\n-1,0) rectangle (\n,1);
  \node at (\n-0.5,0.5) {\i};
}
\end{tikzpicture}
\end{center}
Since $n<\nu$, the optimal first guess is $6$. After the card $3$ is revealed, the reduced deck still has size at most $\nu$, so the player again guesses $6$. Once $6$ is revealed, all remaining cards are guessed in decreasing order. The full sequence of guesses is
\begin{center}
\begin{tikzpicture}[x=0.6cm,y=0.6cm]
\foreach \i [count=\n from 1] in {6,6,5,4,2,1} {
  \draw (\n-1,0) rectangle (\n,1);
  \node at (\n-0.5,0.5) {\i};
}
\end{tikzpicture}
\end{center}
Thus, the player makes $5$ correct guesses.
\end{example}

\begin{example}[Optimal strategy for $p=0.20$ when $n>\nu$]
Again let $p=0.20$, so $\nu=8$, and now take $n=12>\nu$. Suppose that
\[
(\varepsilon_1,\ldots,\varepsilon_{11})=(1,1,0,0,1,0,1,0,0,0,0).
\]
The corresponding shuffled deck, shown from top to bottom, is
\begin{center}
\begin{tikzpicture}[x=0.6cm,y=0.6cm]
\foreach \i [count=\n from 1] in {1,2,5,7,12,11,10,9,8,6,4,3} {
  \draw (\n-1,0) rectangle (\n,1);
  \node at (\n-0.5,0.5) {\i};
}
\end{tikzpicture}
\end{center}
Since $n>\nu$, the optimal first guess is $1$. While the reduced deck has size greater than $\nu$, the player guesses the immediate successor of the previously revealed card; once the card $5$ is revealed, the reduced deck has size $7$, and the strategy switches to guessing $12$ until $12$ appears. The full sequence of guesses is
\begin{center}
\begin{tikzpicture}[x=0.6cm,y=0.6cm]
\foreach \i [count=\n from 1] in {1,2,3,12,12,11,10,9,8,6,4,3} {
  \draw (\n-1,0) rectangle (\n,1);
  \node at (\n-0.5,0.5) {\i};
}
\end{tikzpicture}
\end{center}
Thus, the player makes $10$ correct guesses.
\end{example}

As there are several results, we find it convenient to first give a summary before we set up the necessary notations and state our results formally. 

\subsection*{Summary of our results}
To the best of our knowledge, this is the first work addressing the complete feedback game for the asymmetric shelf shuffle. We provide a comprehensive analysis of the $X_n$, the number of correct guesses under an optimal strategy. 

In Theorem~\ref{thm:MeanAndVariance}, we provide the exact and explicit formulae for the mean and the variance of $X_n$. Theorem~\ref{thm:DistributionXn} gives an explicit description of the law of $X_n$. We also prove that the random variable $X_n$ is log-concave for all $p\in (0, 1)$ and $n\geq 1$ (Theorem~\ref{thm:logConcave}). Also, it is shown that (with suitable centering and scaling), the random variable $X_n$ satisfies a central limit theorem (CLT) in Theorem~\ref{thm:CLT} and a local CLT in Theorem~\ref{thm:LocalCLT}. We also prove a large deviation principle for $X_n/n$ with speed $n$ and an explicit good rate function in Theorem~\ref{thm:LDP}. Finally, we establish two curious phase transition phenomena for $X_n=X_n(p)$ when $p$ is near $1$ and near $0$.

\subsection*{Notation}
For sequences $a_n$ and $b_n$, we write $a_n\ll b_n$ if $a_n/b_n\to 0$ as $n\to\infty$. We denote by $\Var(X)$ the variance of a random variable $X$. For random variables $X$ and $Y$, we write $X\law Y$ if $X$ and $Y$ are equal in law. An indicator is denoted by $\mathbb{I}(\cdot)$. Given a sequence $a_n$, we say that $a_n=O(f(n))$ if there exists a constant $C>0$ such that $|a_n|\leq Cf(n)$ for all large $n$. Sometimes, we will write $a_n=O_t(f(n))$ to emphasize that the constant $C$ depends on a parameter $t$.

\subsection{Setup and Main results}
Unless stated otherwise, we only consider the asymmetric shelf shuffle with a single shelf, that is, $m=1$. When $p\in \{0, 1\}$, the shuffle is deterministic. To avoid trivialities, we always assume that $p\in (0, 1)$, unless stated otherwise. We always assume that the player employs the optimal guessing strategy (described in Proposition~\ref{prop:Strategy}) for guessing. For $p\in (0, 1)$, we use $X_n\equiv X_n(p)$ to denote the total number of correct guesses (under an optimal strategy) in the asymmetric case. We also define
\begin{equation}
\label{eq:defofnu}
q:=1-p, \qquad \nu:= \begin{cases}
1, & p\ge \frac{1}{2},\\
\left\lfloor \frac{\log p}{\log(1-p)}\right\rfloor+1, & 0<p<\frac{1}{2}
\end{cases}.
\end{equation}
\subsubsection{Expectation and variance}
Our first main result is the determination of the \emph{exact} mean and variance of the number of correct guesses. 

\begin{theorem}
\label{thm:MeanAndVariance}
We show that the expected number of correct guesses $X_n$ is given by
\[
\mathbb{E}(X_n)=\begin{cases}
qn+p, & 1\le n\le \nu,\\
(1-pq)n+2p-(\nu+1)p^2-q^\nu, & n\ge \nu+1,
\end{cases}
\]
and the variance $\Var(X_n)$ of $X_n$ is
\[
\begin{cases}
(n-1)pq, & 1\le n\le \nu,\\[1mm]
(\nu-1)pq+\bigl(1-2(\nu-1)p\bigr)q^\nu-q^{2\nu}, & n=\nu+1,\\[1mm]
pq(1-3pq)\,n-2pq+(3\nu+5)p^2q^2+\bigl(1-2\nu p+2p^2\bigr)q^\nu-q^{2\nu},
& n\ge \nu+2.
\end{cases}
\]  
\end{theorem}

For the symmetric shelf-shuffle (when $p=1/2$), we see that the expected number of correct guesses is $3n/4$ as soon as $n\geq 2$, and the variance is $n/16$ as soon as $n\geq 3$. We record the following corollary when $p=1/2$. 
\begin{corollary}[The mean and the variance in symmetric case]
When $p=1/2$, we have 
\[
\E[X_n] =\begin{cases}
    1, & n=1, \\
    \frac{3n}{4}, & n\geq 2
\end{cases},
\]
and 
\[
\Var(X_n) = \begin{cases}
    0, & n=1, \\
    \frac{1}{4}, & n=2, \\
    \frac{n}{16}, & n\geq 3\,.
\end{cases}
\]
\end{corollary}

\subsubsection{Distribution of \texorpdfstring{$X_n$}{Xn}}
Perhaps surprisingly, we find the explicit distribution of $X_n$. 
\begin{theorem}[Exact distribution of $X_n$]
\label{thm:DistributionXn}
Let $\pi_{n,k}:=\mathbb{P}(X_n=k)$. Then, the following hold:
\begin{enumerate}
\item If $p\ge \frac{1}{2}$, then 
\[
\pi_{n,k}
=
\sum_{m=n-k}^{k}
\binom{m-1}{n-k-1}\binom{n-m}{n-k}\,q^m p^{\,n-1-m},
\]
for $\lceil n/2\rceil\le k\le n-1$ and $\pi_{n,n}=p^{\,n-1}$. And, $\pi_{n,k}=0$ otherwise.

\item If $0<p<\frac12$ and $1\le n\le \nu$, then
\[
\pi_{n,k}
=
\binom{n-1}{k-1}q^{\,k-1}p^{\,n-k},
\qquad 1\le k\le n,
\]
and $\pi_{n,k}=0$ otherwise.

\item If $0<p<\frac12$ and $n>\nu$, then
\[
\pi_{n,k}
=
p\,\pi_{n-1,k-1}
+\sum_{j=2}^{n-1}pq^{j-1}\pi_{n-j,k-j+1}
+q^{n-1}\mathbf{1}_{\{k=n-1\}},
\]
with the initial condition
\[
\pi_{m,k}
=
\binom{m-1}{k-1}q^{\,k-1}p^{\,m-k},
\qquad
1\le m\le \nu,\ \ 1\le k\le m.
\]
\end{enumerate}
\end{theorem}

As a corollary, we can identify the number of incorrect guesses as a deterministic function of the shelf word. Let $R_0(w)$ denote the number of zero-runs in a binary word $w$, and recall that $\varepsilon_j=1$ if card $j$ is placed at the top of the pile and $\varepsilon_j=0$ otherwise.

\begin{corollary}[Incorrect guesses as a function of the shelf word]
\label{cor:zerorunsdistribution}
The number of incorrect guesses admits the following description.
\begin{enumerate}
\item If $p\ge 1/2$, then
\[
n-X_n=R_0(\varepsilon_1,\ldots,\varepsilon_{n-1}).
\]
In particular, $n-X_n$ has the law of the number of zero-runs in a Bernoulli$(p)$ word of length $n-1$.
\item If $0<p<1/2$ and $1\le n\le \nu$, then
\[
n-X_n=\sum_{j=1}^{n-1}\varepsilon_j.
\]
In particular, $n-X_n \law \operatorname{Bin}(n-1,p)$.
\item If $0<p<1/2$ and $n>\nu$, set $a=n-\nu$ and
\[
\tau=\min\{j\in\{a,a+1,\ldots,n-1\}:\varepsilon_j=1\},
\]
with the convention that $\tau=n$ if the set is empty. Then
\[
n-X_n=R_0(\varepsilon_1,\ldots,\varepsilon_a)+\sum_{j=\tau+1}^{n-1}\varepsilon_j.
\]
The second term is interpreted as $0$ when $\tau=n$.
\end{enumerate}
\end{corollary}
Our next result shows that the distribution of $X_n$ is log-concave. Recall that a sequence of non-negative real numbers $(a_{k})_{k\geq 0}$ is said to be log-concave if 
\[
a_k^2\geq a_{k-1}a_{k+1}\, \quad \text{for all } k\geq 1\;.
\]
For a random variable $Y$ taking values in $\mathbb{N}$, we say that $Y$ is log-concave if the probability mass function $k\mapsto p(k):=\mathbb{P}(Y=k)$ is a log-concave sequence.
\begin{theorem}
\label{thm:logConcave}
For every $p\in (0, 1)$ and $n\geq 1$, the number of correct guesses $X_n\equiv X_n(p)$ is log-concave.
\end{theorem}
Log-concavity of a sequence entails several interesting and useful consequences. While we do not pursue this direction here, we refer the interested readers to~\cite{Pitman, SW14} for a glimpse of applications and consequences. 

\subsubsection{Limit laws}
Our next set of results describes the limit laws for $X_n$. In particular, we show that $X_n$ (suitably centered and scaled) satisfies a central limit theorem (CLT) with convergence rate of order $n^{-1/2}$.
\begin{theorem}[CLT]
\label{thm:CLT}
We have
    \[\mathbb{P}\left(
        \frac{X_n - \mathbb{E}(X_n)}{\sqrt{\mathbb{V}(X_n)}} \le x
        \right)
        =\Phi(x)+O_p\left(n^{-1/2}\right)
    \]
where 
\[
\Phi(x) = \frac{1}{\sqrt{2\pi}}\int_{-\infty}^{x}e^{-t^2/2}\,dt
\]
is the standard normal cumulative distribution function.
\end{theorem}
In particular, when $p=1/2$, we get the following.
\begin{corollary}[CLT in the symmetric case]
Let $p=1/2$ and let $X_n=X_n(1/2)$. Then, 
\[\mathbb{P}\left(
        \frac{X_n - 3n/4}{\sqrt{n/16}} \le x
        \right)
        =\Phi(x)+O_p\left(n^{-1/2}\right)\;.
    \]
\end{corollary}

Our next result gives a local CLT for $X_n$. The central limit theorem only controls the typical behaviour of $X_n$. One is often interested in finer types of fluctuations. Note that a central limit theorem suggests that 
\[
\mathbb{P}(a\leq X_n\leq a+h)\approx \frac{h}{\sigma_n\sqrt{2\pi}}\exp\left(-\frac{(a-\mu_n)^2}{2\sigma_n^2}\right)\;.
\]
However, a CLT does not quite imply this (even if one has a quantitative Berry--Esseen rate), and indeed such approximations may fail even when a CLT holds. The local CLT provides the above approximation for the fluctuation of $X_n$. 
\begin{definition}[{\cite[Definition IX.4]{FlaSed}}]
A sequence of discrete probability distributions $p_{n,k}:=\mathbb{P}(X_n=k)$ with mean $\mu_n$ and standard deviation $\sigma_n$ is said to obey a
\emph{local limit law of the Gaussian type} if, there exists a sequence
$\varepsilon_n \to 0$, 
\begin{equation}
\sup_{x\in\mathbb{R}}
\left|
\sigma_n\, p_{n,\lfloor \mu_n + x\sigma_n \rfloor}
-\frac{1}{\sqrt{2\pi}}e^{-x^2/2}
\right|
\le \varepsilon_n .
\end{equation}
\end{definition}

\begin{theorem}[Local CLT]
\label{thm:LocalCLT}
$X_n$ satisfies a local CLT with rate $\varepsilon_n=O_p(n^{-1/2})$.
\end{theorem}

Our next result exhibits a large deviation principle (LDP) for $X_n$. We refer the readers to the excellent books~\cite{DZ10, Varadhan} for an introduction to the fascinating area of large deviations. Before we state our result, we recall some definitions.

A \emph{good rate function} is a lower semicontinuous function $I:\mathbb{R}\to [0, +\infty]$ such that the level sets
\[
\Psi_{I}(\alpha):=\{x\in \R: I(x)\leq \alpha\}
\]
are compact sets for all $\alpha\in [0, \infty)$. For a set $A\subseteq \R$, we write $\overline{A}$ for the closure of $A$ and $A^{\circ}$ for the interior of $A$. We say that a sequence of random variables $(Y_n)_{n\geq 1}$ satisfies a \emph{large deviation principle} (LDP) with speed $n$ and the good rate function $I$, if 
\begin{align*}
    -\inf_{x\in {A^{\circ}}} I(x)&\leq \liminf_{n\to \infty}\frac{1}{n}\log\mathbb{P}(Y_n\in A)\\
    &\leq \limsup_{n\to \infty}\frac{1}{n}\log\mathbb{P}(Y_n\in A)\leq -\inf_{x\in \overline{A}} I(x),
\end{align*}
for every Borel subset $A\subseteq \R$. Our next result gives an LDP for $X_n\equiv X_n(p)$. 
\begin{theorem}[LDP for $X_n$]
\label{thm:LDP}
Let $p\in (0, 1)$. Set 
\[y_p(x) := \frac{1-x+\sqrt{(1-x)^2+(1-4pq)(2x-1)}}{2x-1}, \quad x\in (1/2, 1)\;.
\]
Define the good rate function $I_p$ as
\begin{align*}
    I_p(x) :=\begin{cases}
        (1-x)\log\frac{y_p^2(x)-(1-4pq)}{4pq}+\log \frac{2}{1+y_p(x)}, & 1/2< x <1\\
        +\infty, & x\not\in [1/2, 1]\;,
    \end{cases}
\end{align*}
and continuously extend $I_p$ to $1/2$ and $1$ by 
\[
I_p(1/2)=\frac{1}{2}\log\frac{1}{pq}, \qquad I_p(1) = \log\frac{1}{\max\{p, 1-p\}}\;. 
\]
Then, $X_n/n$ satisfies an LDP with speed $n$ and the good rate function $I_p$.
\end{theorem}

\begin{corollary}
When $p=1/2$, the random variable $X_n/n$ satisfies an LDP with speed $n$ and good rate function 
    \[
I_{1/2}(x)= (2x-1)\log (2x-1)+2(1-x)\log (1-x)+(3-2x)\log 2,
    \]
when $x\in [1/2, 1]$ and $I_{1/2}(x)=+\infty$ otherwise.
\end{corollary}

\subsubsection{Phase transitions}
We now describe phase transitions near $p=1$ and $p=0$ respectively. Note that when $p\in \{0, 1\}$, the final deck is deterministic and therefore $X_n=n$ almost surely. 

We begin with the phase transition at $p=1$. To this end, it is instructive to analyze $X_n(p)$ when $p\approx 1$. Let $\alpha,\lambda>0$ and let  $p=1-\frac{\lambda}{n^{\alpha}}$. In this case, 
\[\P\{X_n=n\}=\round{1-\frac{\lambda}{n^{\alpha}}}^{n-1}.\]
As $n\to \infty$, we get
\[
\P(X_n=n)\to
\begin{cases}
0,\quad 0<\alpha<1,\\
e^{-\lambda},\quad \alpha=1,\\
1,\quad \alpha> 1,
\end{cases}
\]
which follows directly from the explicit formula $\P(X_n=n)=p^{n-1}$ and the $\exp-\log$ expansion. In this regime, it is natural to investigate the fluctuation of $X_n/n$ or $n-X_n$. This is what our next result describes.
\begin{theorem}
\label{thm:PhaseTransition}
Let $X_n$ be as above and let $p=1-\frac{\lambda}{n^{\alpha}}$.
\begin{enumerate}
    \item When $\alpha>1$, $X_n/n$ converges to $1$ in probability.
    \item  When $\alpha=1$, (i.e. $p=1-\lambda/n$), we show that the number of incorrect guesses $D_n:=n-X_n$ converges to a $\text{Poisson}(\lambda)$ random variable as $n\to \infty$.
    \item For $\alpha\in(0,1)$, we have that $(X_n-\mathbb{E}(X_n))/\sqrt{\Var(X_n)}$ converges to a normal distribution. 
\end{enumerate}
\end{theorem}
We now describe the phase transition in the $p=p_n\to 0$ regime. This is motivated by the two different distributions for $X_n$ in the case where $0<p<1/2$ described in Corollary~\ref{cor:zerorunsdistribution}.
Notice that Corollary~\ref{cor:zerorunsdistribution} predicts binomial-like behavior for $X_n$ when $n\leq \nu_p$. We capture this in the following theorem. 
\begin{theorem}
\label{thm:phasetransitionp0}
Suppose that $p\to 0$ as $n\to\infty$. 
    \begin{enumerate}
        \item If
            \[\frac{1}{n}\ll p\ll\frac{\ln n-\ln\ln n}{n}\]
            then 
            \begin{equation}
            \label{eq:normalconvergencep0}
                \frac{X_n-nq}{\sqrt{npq}}
            \end{equation}
            converges in distribution to a standard normal random variable.
        \item If $p\ll 1/n$, then $X_n/n\to 1$ in probability.
        \item If $p\gg\frac{\ln n-\ln\ln n}{n}$, then, centered by its mean and variance, $X_n$ has a limiting standard normal distribution.
    \end{enumerate}
\end{theorem}


\section{Optimal strategy and generating function for \texorpdfstring{$X_n$}{Xn}}
\label{sec:generatingfunction}
Recall our discussion preceeding Proposition~\ref{prop:Strategy}. This discussion suggests that given the first card $\fc_n=t$, we get the last $(t-1)$ cards correct. Furthermore, the first guess is correct if and only if $\fc_n=\maxP_n$. And, conditional on $\fc_n=t$, the remaining $(n-t)$ cards are distributed according to an asymmetric single-shelf shuffle. This leads to the following stochastic recurrence for the score $X_n$.
\begin{theorem}
\label{thm:DistAssymetric}
Let $p\in (0, 1)$ be fixed. Let $X_n=X_n(p)$ be the number of correctly guessed cards, starting with a deck of $n$ cards, after an asymmetric one-time single-shelf shuffle with parameter $p \in(0,1)$. Then, $X_n$ satisfies the following recurrence for $n\geq 2$:
\begin{equation}
\label{eqn:Xn_DistEqn3}
X_n \law X^{\ast}_{n-\fc_n}+\fc_n-1 + \mathbb{I}_n(\fc_n=\maxP_n), 
\end{equation}
with initial values $X_1=1$ and $X_0=0$, where $\fc_n$ denotes the first card in a shuffled deck.
Here, $\mathbb{I}_n(\fc_n=\maxP_n)$ denotes the indicator of the event that the first guess is correct, and $X^{\ast}$ denotes an independent copy of $X$.
\end{theorem}
\begin{proof}
Notice that once the first card $\fc_n$ is shown, the cards $\{1,2,\ldots,\fc_{n}-1\}$ will be at the bottom of the deck in descending order, which corresponds to $\fc_n-1$ guesses that are guaranteed to be correct. Then, we know that we are left with the $n-\fc_n$ cards $\{\fc_n +1,\ldots,n\}$ in unknown positions with the next guess. By definition, these cards must be distributed as an asymmetric shelf-shuffle independent of the first card, and the equality in distribution follows.
\end{proof}
This recurrence relation plays a foundational role in our analysis, as it yields the generating function of the number of correct guesses under the optimal strategy. All of our main results in Section \ref{sec:Proofs} crucially rely on the careful analysis of this generating function. To make use of Theorem~\ref{thm:DistAssymetric}, we also need a precise description of the distribution of $\fc_n$. We record it as a lemma below.
\begin{lem}
\label{lem:firstcard}
The distribution of the first card $\fc_n$ in a single asymmetric shelf shuffle with parameter $p\in(0,1)$ is given as follows:
\begin{equation}
\label{eq:fc}
\P(\fc_n=i)=p(1-p)^{i-1},\quad 1\le i \le n-1
\end{equation}
and 
\[
\P(\fc_n=n)=(1-p)^{n-1}.
\]
\end{lem}
\begin{proof}
    Fix $i\in[n-1]$. Recalling Proposition \ref{prop:Strategy}, for card $i$ to be at the top of the shuffled deck, we need every card in $[i-1]$ to be sent to the bottom of the pile and card $i$ to be sent to the top. This happens with probability $p(1-p)^{i-1}$. For card $n$ to be at the top, we need all the cards in $[n-1]$ to be sent to the bottom, which happens with probability $(1-p)^{n-1}$.
\end{proof}

\begin{corollary}[A recurrence for the generating function]
   Let $s_n(v)=\E[v^{X_n}]$ denote the generating function of $X_n$. Then, $s_0(v)=1$, $s_1(v)=v$, and for $2\le n\le \nu$ we have
\begin{equation}
\label{eq:asym-rec-small}
s_n(v)
=
p\,s_{n-1}(v)
+\sum_{j=2}^{n-1}pq^{j-1}v^{j-1}s_{n-j}(v)
+q^{n-1}v^n,
\end{equation}
whereas for $n>\nu$ we have
\begin{equation}
\label{eq:asym-rec-large}
s_n(v)
=
pv\,s_{n-1}(v)
+\sum_{j=2}^{n-1}pq^{j-1}v^{j-1}s_{n-j}(v)
+q^{n-1}v^{n-1}.
\end{equation}

\end{corollary}

Given a sequence of random variables $(X_n)_{n\geq 1}$, we define the generating function $s_n(v)$ of $X_n$ and the bivariate generating function $S(z, v)$ as follows:
\[s_n(v):=\E[v^{X_n}], \qquad S(z,v):=\sum_{n\geq 1}\E[v^{X_n}]z^n.\] 
We further write
\[
S_{\le \nu}(z,v):=\sum_{n=1}^{\nu}s_n(v)z^n,\qquad
S_{>\nu}(z,v):=\sum_{n>\nu}s_n(v)z^n.
\]

\begin{proposition}[Unified bivariate generating function]
\label{prop:asym-bgf}
For $1\le n\le \nu$ we have
\[
s_n(v)=v\bigl(p+qv\bigr)^{n-1},
\]
and 
\[
s_{\nu+1}(v)=v^2(p+qv)^{\nu-1}+q^\nu v^\nu(1-v).
\]
Therefore
\[
S_{\le \nu}(z,v)
=\sum_{n=1}^{\nu}s_n(v)z^n
=
\frac{zv\bigl(1-z^\nu(p+qv)^\nu\bigr)}{1-z(p+qv)}.
\]
Moreover, the full bivariate generating function is then given by
\begin{equation}
\label{eq:asym-bgf-unified}
S(z,v)
=
S_{\le \nu}(z,v)
+
\frac{z^{\nu+1}\Bigl(s_{\nu+1}(v)+pq\,v(1-v)z\,s_\nu(v)\Bigr)}
     {1-vz+pq\,v(v-1)z^2}.
\end{equation}
\end{proposition}

\begin{proof}
It can be verified by direct computation that 
\[
s_n(v)=v(p+qv)^{n-1},\qquad 1\le n\le \nu,
\]
satisfies the reccurence~\eqref{eq:asym-rec-small}.
This proves the formula for $s_n(v)$ and, by summing the resulting geometric series,
\[
S_{\le \nu}(z,v)=\sum_{n=1}^{\nu}v(p+qv)^{n-1}z^n
=
\frac{zv\bigl(1-z^\nu(p+qv)^\nu\bigr)}{1-z(p+qv)}.
\]
Substituting $s_1(v),\dots,s_\nu(v)$ into
\eqref{eq:asym-rec-large} with $n=\nu+1$ gives
\[
\begin{aligned}
s_{\nu+1}(v)
=v^2(p+qv)^{\nu-1}+q^\nu v^\nu(1-v).
\end{aligned}
\]

We now derive a universal second-order recurrence for the tail. For $n\ge \nu+2$,
multiply \eqref{eq:asym-rec-large} with $n-1$ in place of $n$ by $qv$ and subtract. This yields the following recurrence:
\begin{equation}
\label{eq:asym-rec-second-order}
s_n(v)=v\,s_{n-1}(v)+pq\,v(1-v)\,s_{n-2}(v),\qquad n\ge \nu+2.
\end{equation}
Since $S_{>\nu}(z,v)- s_{\nu+1}(v)z^{\nu+1}= 
\sum_{n\ge \nu+2}s_n(v)z^n$, 
we get
\[
\begin{aligned}
S_{>\nu}(z,v)-s_{\nu+1}(v)z^{\nu+1}
&=
vz\sum_{n\ge \nu+2}s_{n-1}(v)z^{n-1}
+pq\,v(1-v)z^2\sum_{n\ge \nu+2}s_{n-2}(v)z^{n-2}\\
&=
vz\,S_{>\nu}(z,v)
+pq\,v(1-v)z^2\Bigl(S_{>\nu}(z,v)+s_\nu(v)z^\nu\Bigr).
\end{aligned}
\]
Rearranging yields
\[
\bigl(1-vz+pq\,v(v-1)z^2\bigr)S_{>\nu}(z,v)
=
z^{\nu+1}\Bigl(s_{\nu+1}(v)+pq\,v(1-v)z\,s_\nu(v)\Bigr).
\]
\end{proof}

\begin{corollary}[The case $p\ge \frac{1}{2}$]
\label{cor:asym-bgf-upper}
If $p\ge \frac12$, then $\nu=1$ and
\[
S(z,v)
=
\frac{zv\bigl(1+q(1-v)z\bigr)}{1-vz+pq\,v(v-1)z^2}.
\]
\end{corollary}

\begin{proof}
If $p\ge \frac12$, then $\nu=1$, so
\[
S_{\le 1}(z,v)=zv,\qquad s_1(v)=v,\qquad s_2(v)=qv+pv^2.
\]
Substituting these into \eqref{eq:asym-bgf-unified} gives the desired result.
\end{proof}

\begin{corollary}[The initial regime $0<p<\frac12$, $n\le \nu$]
\label{cor:asym-bgf-initial}
If $0<p<\frac12$, then
\[
S_{\le \nu}(z,v)
=
\frac{zv\bigl(1-z^\nu(p+qv)^\nu\bigr)}{1-z(p+qv)}.
\]
Equivalently,
\[
X_n \law 1+\operatorname{Bin}(n-1,q),
\qquad 1\le n\le \nu.
\]
\end{corollary}

\begin{proof}
The first assertion is contained in Proposition~\ref{prop:asym-bgf}. Since the
probability generating function of $1+\operatorname{Bin}(n-1,q)$ is
$v(p+qv)^{n-1}$, the distributional identity follows.

\end{proof}

\begin{proposition}[Probability generating functions $s_n(v)$]
\label{prop:asym-pgf}
Let
\[
\Delta(v):=v^2-4pq\,v(v-1),
\qquad
\lambda_\pm(v):=\frac{v\pm \sqrt{\Delta(v)}}{2},
\]
where the branch of $\sqrt{\Delta(v)}$ is chosen so that $\sqrt{\Delta(1)}=1$.
Note that $\lambda_\pm(v)$ are the two roots of
\[
\lambda^2-v\lambda+pq\,v(v-1)=0.
\]

\begin{enumerate}
\item If $p\ge \frac12$, then for every $n\ge 1$,
\[
s_n(v)
=
\frac{\bigl(s_2(v)-\lambda_-(v)s_1(v)\bigr)\lambda_+(v)^{\,n-1}
+\bigl(\lambda_+(v)s_1(v)-s_2(v)\bigr)\lambda_-(v)^{\,n-1}}
{\lambda_+(v)-\lambda_-(v)},
\]
where
\[
s_1(v)=v,\qquad s_2(v)=qv+pv^2.
\]

\item If $0<p<\frac12$ and $1\le n\le \nu$, then
\[
s_n(v)=v(p+qv)^{n-1}.
\]

\item If $0<p<\frac12$ and $n\ge \nu$, then
\[
s_n(v)
=
\frac{\bigl(s_{\nu+1}(v)-\lambda_-(v)s_\nu(v)\bigr)\lambda_+(v)^{\,n-\nu}
+\bigl(\lambda_+(v)s_\nu(v)-s_{\nu+1}(v)\bigr)\lambda_-(v)^{\,n-\nu}}
{\lambda_+(v)-\lambda_-(v)},
\]
where
\[
s_\nu(v)=v(p+qv)^{\nu-1},
\qquad
s_{\nu+1}(v)=v^2(p+qv)^{\nu-1}+q^\nu v^\nu(1-v).
\]
Equivalently, for $n\ge \nu+2$,
\[
s_n(v)=v\,s_{n-1}(v)+pq\,v(1-v)\,s_{n-2}(v).
\]
\end{enumerate}
\end{proposition}
\begin{proof}
Part (2) is contained in Proposition~\ref{prop:asym-bgf}. For $n\ge \nu+2$,
Proposition~\ref{prop:asym-bgf} gives the second-order recurrence
\[
s_n(v)-v\,s_{n-1}(v)-pq\,v(1-v)\,s_{n-2}(v)=0.
\]
Its characteristic polynomial is
\[
\lambda^2-v\lambda+pq\,v(v-1)=0,
\]
with roots $\lambda_\pm(v)$. Hence
\[
s_n(v)=A(v)\lambda_+(v)^{\,n-\nu}+B(v)\lambda_-(v)^{\,n-\nu},
\qquad n\ge \nu,
\]
for suitable functions $A(v),B(v)$. Solving
\[
A(v)+B(v)=s_\nu(v),
\qquad
A(v)\lambda_+(v)+B(v)\lambda_-(v)=s_{\nu+1}(v)
\]
gives the stated formula. The case $p\ge \frac12$ is the specialization $\nu=1$.
\end{proof}

\section{Proofs}
\label{sec:Proofs}
This section contains the proofs of the main results in the introduction.
\subsection{Mean, Variance, and Exact Distribution of \texorpdfstring{$X_n$}{Xn}} 
In this subsection, we find the mean and variance of the number of correct guesses $X_n$, hence establishing Theorem \ref{thm:MeanAndVariance}. We also find the exact distribution of $X_n$, proving Theorem \ref{thm:DistributionXn}. We begin by finding the mean and variance using the generating functions from Section \ref{sec:generatingfunction}.
\begin{proof}[Proof of Theorem \ref{thm:MeanAndVariance}]
Set
\[
\mu_n:=\mathbb{E}(X_n),\qquad \sigma_n^2:=\Var(X_n).
\]
For $1\le n\le \nu$, Corollary~\ref{cor:asym-bgf-initial} yields
\[
X_n\overset{L}{=}1+\operatorname{Bin}(n-1,q),
\]
hence
\[
\mu_n=1+(n-1)q=qn+p,
\qquad
\sigma_n^2=(n-1)pq.
\]

Next, using the explicit expression for $s_{\nu+1}(v)$ from
Proposition~\ref{prop:asym-bgf}, we obtain
\[
\mu_{\nu+1}=s_{\nu+1}'(1)=1+p+\nu q-q^\nu.
\]
A short simplification shows that
\[
\mu_{\nu+1}=(1-pq)(\nu+1)+2p-(\nu+1)p^2-q^\nu.
\]
For $n\ge \nu+2$, differentiate the recurrence
\[
s_n(v)=v\,s_{n-1}(v)+pq\,v(1-v)\,s_{n-2}(v)
\]
at $v=1$. Since $s_m(1)=1$ for all $m$, we get
\[
\mu_n=\mu_{n-1}+1-pq,\qquad n\ge \nu+2.
\]
Thus
\[
\mu_n=\mu_{\nu+1}+(n-\nu-1)(1-pq),\qquad n\ge \nu+1,
\]
which is exactly the stated formula for $\mathbb{E}(X_n)$.

For the variance, set
\[
f_n:=s_n''(1)=\mathbb{E}\!\bigl(X_n(X_n-1)\bigr).
\]
Differentiating the explicit formula for $s_{\nu+1}(v)$ twice gives
\[
f_{\nu+1}
=
2+4(\nu-1)q+(\nu-1)(\nu-2)q^2-2\nu q^\nu.
\]
Therefore
\[
\sigma_{\nu+1}^2
=
f_{\nu+1}+\mu_{\nu+1}-\mu_{\nu+1}^2
=
(\nu-1)pq+\bigl(1-2(\nu-1)p\bigr)q^\nu-q^{2\nu}.
\]
Now differentiate the recurrence twice and evaluate at $v=1$. This yields
\[
f_n=f_{n-1}+2\mu_{n-1}-2pq\bigl(1+\mu_{n-2}\bigr),
\qquad n\ge \nu+2.
\]
Using this identity together with
\[
\mu_n-\mu_{n-1}=1-pq,
\]
one finds first
\[
\sigma_{\nu+2}^2
=
\nu pq-p^2q^2+\bigl(1-2\nu p+2p^2\bigr)q^\nu-q^{2\nu},
\]
and then, for $n\ge \nu+3$,
\[
\begin{aligned}
\sigma_n^2-\sigma_{n-1}^2
&=
(f_n-f_{n-1})+(\mu_n-\mu_{n-1})-(\mu_n^2-\mu_{n-1}^2)\\
&=
pq(1-3pq).
\end{aligned}
\]
Hence
\[
\sigma_n^2=\sigma_{\nu+2}^2+(n-\nu-2)pq(1-3pq),\qquad n\ge \nu+2,
\]
which simplifies to the stated third line.

Finally, when $p\ge \frac12$, we have $\nu=1$. Substituting $\nu=1$ into the general
formulae gives
\[
\mathbb{E}(X_1)=1,\qquad
\mathbb{E}(X_n)=(1-p+p^2)n+3p-1-2p^2,\qquad n\ge 2,
\]
and
\[
\Var(X_1)=0,\qquad \Var(X_2)=pq,
\]
\[
\Var(X_n)=pq(1-3pq)\,n+pq(10p-8p^2-3),\qquad n\ge 3.
\]
\end{proof}
Next, we find the exact distribution of $X_n$.
\begin{proof}[Proof of Theorem \ref{thm:DistributionXn}]
For $0<p<\frac12$ and $1\le n\le \nu$, part (2) is exactly
Corollary~\ref{cor:asym-bgf-initial}.

For $0<p<\frac12$ and $n>\nu$, coefficient extraction from
\eqref{eq:asym-rec-large} gives
\[
\pi_{n,k}
=
p\,\pi_{n-1,k-1}
+\sum_{j=2}^{n-1}pq^{j-1}\pi_{n-j,k-j+1}
+q^{n-1}\mathbf{1}_{\{k=n-1\}},
\]
which proves the last part. It remains to prove the first part. 

Assume that $p\ge \frac12$, and set $D_n:=n-X_n$. By Corollary~\ref{cor:asym-bgf-upper},
\[
R(z,v):=\sum_{n\ge 1}\mathbb{E}\!\left(v^{D_n}\right)z^n
=
S(vz,v^{-1})
=
z\,\frac{1-q(1-v)z}{1-z+pq(1-v)z^2}.
\]
Thus, $D_n$ is encoded by the usual bivariate generating function for the number of
zero-runs in a Bernoulli$(p)$ word of length $n-1$, with
\[
\mathbb{P}(1)=p,\qquad \mathbb{P}(0)=q.
\]
We now count such words directly. Fix $j\ge 1$, and suppose the word has exactly
$m$ zeros and exactly $j$ zero-runs. Then the $m$ zeros must be split into $j$
positive blocks, which can be done in
\[
\binom{m-1}{j-1}
\]
ways. The remaining $n-1-m$ ones create
\[
(n-1-m)+1=n-m
\]
gaps, and one chooses $j$ of these gaps in which to place the zero-blocks. This gives
\[
\binom{n-m}{j}
\]
possibilities. Each such word has probability $q^m p^{n-1-m}$, so
\[
\mathbb{P}(D_n=j)
=
\sum_{m=j}^{n-j}
\binom{m-1}{j-1}\binom{n-m}{j}\,q^m p^{n-1-m}.
\]
Since $D_n=n-X_n$, we set $j=n-k$ and obtain, for $\lceil n/2\rceil\le k\le n-1$,
\[
\mathbb{P}(X_n=k)
=
\sum_{m=n-k}^{k}
\binom{m-1}{n-k-1}\binom{n-m}{n-k}\,q^m p^{n-1-m}.
\]
Finally, the case $D_n=0$ corresponds to the all-ones word, hence
\[
\mathbb{P}(X_n=n)=\mathbb{P}(D_n=0)=p^{n-1}.
\]
\end{proof}

\begin{proof}[Proof of Corollary~\ref{cor:zerorunsdistribution}]
Fix the shelf word $(\varepsilon_1,\ldots,\varepsilon_{n-1})$. This word determines the shuffled deck: the indices with $\varepsilon_j=1$ appear before $n$ in increasing order, and the indices with $\varepsilon_j=0$ appear after $n$ in decreasing order.

If $p\ge 1/2$, the strategy guesses the immediate successor of the previously revealed card until $n$ or $n-1$ is revealed. Thus, an incorrect guess occurs exactly when the next revealed card jumps over a non-empty maximal block of zeros in the shelf word. Each zero-run causes exactly one such jump, and no other incorrect guesses occur. Hence
\[
n-X_n=R_0(\varepsilon_1,\ldots,\varepsilon_{n-1}).
\]

If $0<p<1/2$ and $n\le\nu$, the strategy guesses $n$ until $n$ or $n-1$ is revealed. Therefore, each card appearing before $n$ gives one incorrect guess, and these cards are precisely those $j\le n-1$ for which $\varepsilon_j=1$. This gives
\[
n-X_n=\sum_{j=1}^{n-1}\varepsilon_j.
\]

Finally, assume that $0<p<1/2$ and $n>\nu$, and set $a=n-\nu$. While the remaining deck has size greater than $\nu$, the strategy guesses the immediate successor of the previously revealed card. Hence, this first part of the play contributes exactly the number of zero-runs in $(\varepsilon_1,\ldots,\varepsilon_a)$. The card at which the remaining deck first has size at most $\nu$ is indexed by
\[
\tau=\min\{j\in\{a,a+1,\ldots,n-1\}:\varepsilon_j=1\},
\]
with the convention $\tau=n$ if the set is empty. From that point on, the strategy guesses $n$ until $n$ appears, so the remaining incorrect guesses are precisely the later cards before $n$, namely the indices $j>\tau$ with $\varepsilon_j=1$. This proves
\[
n-X_n=R_0(\varepsilon_1,\ldots,\varepsilon_a)+\sum_{j=\tau+1}^{n-1}\varepsilon_j.
\]
\end{proof}
\subsection{Log-Concavity}
In this subsection, we prove Theorem \ref{thm:logConcave}. It is convenient to work with
\[
D_n:=n-X_n, \qquad d_n(v):=\mathbb{E}\!\left[v^{D_n}\right]=v^ns_n(v^{-1}).
\]
It is easy to see that the log-concavity of $D_n$ is equivalent to the log-concavity of $X_n$. In this section, we prove that $D_n$ is log-concave. Before proving the result, we state two important lemmas for log-concave sequences. The first lemma is well-known, see e.g. \cite{StanleyLogConc}.
\begin{lem}
\label{lem:binomlogconcave}
    The sequence of binomial coefficients is log-concave, that is,
    \[\binom{n}{k}^2\geq\binom{n}{k-1}\binom{n}{k+1}.\]
\end{lem}
Using this, we immediately get that $D_n=n-X_n$ is log concave when $p\in (0, 1/2)$ and $n\leq \nu$. For completeness, we record the short proof below. 
\begin{lem}
\label{lem:logConcavitysmallpsmalln}
   Fix $0<p<1/2$ and $1\leq n\leq\nu$. Then, $X_n$ is log-concave.
\end{lem}
\begin{proof}
From Theorem \ref{thm:DistributionXn}, we have, denoting $\pi_{n,k}=\P(X_n=k)$,
    \begin{equation}
    \label{eq:p12ineq}
        \begin{aligned}
            (\pi_{n,k})^2 &=\binom{n-1}{k-1}^2 q^{2k-2}p^{2n-2k}\\
            &\geq\binom{n-1}{k-2}\binom{n-1}{k}q^{2k-2}p^{2n-2k}\\
            &=\binom{n-1}{k-2}q^{k-2}p^{n-(k-1)}\binom{n-1}{k}q^{k}p^{n-(k+1)}\\
            &=\pi_{n,{k+1}}\pi_{n,k-1}
        \end{aligned}
    \end{equation}
    where the inequality~\eqref{eq:p12ineq} follows from Lemma \ref{lem:binomlogconcave}. This shows log-concavity in the $0<p<1/2$ and $n\leq\nu$ case.
\end{proof}

We now prove the log-concavity of $D_n$ for $n>\nu$ and $p\in (0, 1)$. The proof uses induction combined with some results from~\cite{GMTW}. Before we begin the proof, we need some notations and definitions.  Fix $p\in (0, 1)$ and $r=\nu-1$ and $B_r\sim \mathrm{Bin}(r, p)$. Define 
\[
V_1(v) = \E[v^{B_r}] = (q+pv)^r, \qquad V_0(v) = \E[v^{(B_r-1)_{+}}]\;. 
\]
Define the polynomials $F_{m}^{(1)}$ and $F_m^{(0)}$ recursively as follows: 
\[
F_0^{(1)}(v) = V_{1}(v), \qquad F_0^{(0)}(v) = V_{0}(v), 
\]
and for $m\geq 0$, 
\begin{align*}
    F_{m+1}^{(1)}(v) = pF_m^{(1)}(v)+qvF_m^{(0)}(v), \quad F_{m+1}^{(0)}(v) = pF_m^{(1)}(v)+qF_m^{(0)}(v).
\end{align*}
With this setup, we rewrite the generating function of $D_n$ for $n>\nu$ in terms of $F_{n-\nu}^{(1)}$. 
\begin{lem}
Fix $p\in (0, 1)$ and $n>\nu$. Then,
    \[d_n(v):=\E\left[v^{D_n}\right] = F_{n-\nu}^{(1)}(v)\;. \]
\end{lem}
\begin{proof}
    We prove this by induction on $n$. There are two base cases to check. First, assume that $n=\nu$. Then,
    \begin{equation}
    \begin{aligned}
        d_{\nu}(v)&=v^{\nu}s_{\nu}(v^{-1})\\
        &=v^{\nu}(v^{-1}(p+qv^{-1})^{\nu-1})\\
        &=(v(p+qv^{-1}))^{\nu-1}\\
        &=(q+pv)^{\nu-1}\\
        &=V_1(v)=F_0^{(1)}(v)
    \end{aligned}
    \end{equation}
    where the second equality follows from Proposition $\ref{prop:asym-pgf}(2)$. 

    Next, assume that $n=\nu+1$. Then, using Proposition $\ref{prop:asym-pgf}$ again, we have
    \begin{equation}
        \begin{aligned}
            d_{\nu+1}(v)&=v^{\nu+1}s_{\nu+1}(v^{-1})\\
            &=v^{\nu+1}(v^{-2}(p+qv^{-1})^{\nu-1}+q^\nu v^{-\nu}(1-v^{-1}))\\
            &=(q+pv)^{\nu-1}+q^{\nu}(v-1).
        \end{aligned}
    \end{equation}
    On the other hand,
    \begin{equation}
        \begin{aligned}
            F_1^{(1)}(v)&=pF^{(1)}_0(v)+qvF^{(0)}_1(v)\\
            &=p(q+pv)^{\nu-1}+qv\E[v^{(B_r-1)_{+}}]\\
            &=p(q+pv)^{\nu-1}+qv(q^{\nu-1}+v^{-1}(q+pv)^{\nu-1}-q^{\nu-1}v^{-1})\\
            &=p(q+pv)^{\nu-1}+q^{\nu}v+q(q+pv)^{\nu-1}-q^{\nu}\\
            &=(q+pv)^{\nu-1}+q^{\nu}(v-1)\\
            &=d_{\nu+1}(v).
        \end{aligned}
    \end{equation}
    Now, we address the induction step. The recurrence for $s_n(v)$ in Equation $\eqref{eq:asym-rec-second-order}$ implies that
    \begin{equation}
    \label{eq:recurrfordn}
    \begin{aligned}
        d_n(v)&=v^ns_n(v^{-1})\\
        &=v^n(v^{-1}\,s_{n-1}(v^{-1})+pq\,v^{-1}(1-v^{-1})\,s_{n-2}(v^{-1}))\\
        &=v^{n-1}s_{n-1}(v^{-1})+pqv^{n-1}s_{n-2}(v^{-1})-pqv^{n-2}s_{n-2}(v^{-1})\\
        &=d_{n-1}(v)+pq(v-1)d_{n-2}(v).
    \end{aligned}
    \end{equation}
Let $n\geq\nu+2$, and assume for sake of induction that $d_k(v)=F_{k-\nu}^{(1)}(v)$ for all $k\leq n$. Using the recurrence from the definitions of $F^{(1)}$ and $F^{(0)}$ and the shorthand $F:=F(v)$, we compute
\begin{equation}
\begin{aligned}
    F^{(1)}_{n-\nu+1}&=pF^{(1)}_{n-\nu}+qvF^{(0)}_{n-\nu}\\
    &=pF^{(1)}_{n-\nu}+qv(pF^{(1)}_{n-\nu-1}+qF^{(0)}_{n-\nu-1})\\
    &=pF^{(1)}_{n-\nu}+qpvF^{(1)}_{n-\nu-1}+q(F^{(1)}_{n-\nu}-pF^{(1)}_{n-\nu-1})\\
    &=pd_n(v)+qpvd_{n-1}(v)+qd_n(v)-pqd_{n-1}(v)\\
    &=d_n(v)+pq(v-1)d_{n-1}(v)\\
    &=d_{n+1}(v),
    \end{aligned}
\end{equation}
where the fourth equality used the induction hypothesis and the last equality used the recurrence for $d_n$ in Equation $\eqref{eq:recurrfordn}$. This completes the proof.
\end{proof}
Let $A_m$ and $B_m$ denote the coefficient sequences of $F_{m}^{(1)}$ and $F_m^{(0)}$, respectively. Theorem~\ref{thm:logConcave} (in $p\in (0, 1)$ and $n>\nu$) follows if we show that $A_m$ is log-concave for $m\geq 1$. We establish this via an inductive argument combined with some results in~\cite{GMTW}. We start with some preparation. 

In the following discussion, we extend all finite sequences to $\mathbb{Z}$ by setting them to zero outside their support. Let $A=(a_n)_{n\in \mathbb{Z}}$ and $B=(b_n)_{n\in \mathbb{Z}}$ be log-concave sequences. We say that $A$ and $B$ are \emph{weakly synchronized}, denoted $A\sim_w B$, if 
\[
a_{k-1}b_{k+1}+a_{k+1}b_{k-1}\leq 2a_kb_k, \quad \text{for all }k\in \mathbb{Z}\;.
\]
Let $S$ denote the right-shift operator on sequences $(SA)_k=a_{k-1}$.
It is easy to verify that if $A, B$ are log-concave and have no internal zeros, then
\begin{equation}
\label{eqn:WeakSynchronizationFacts}
  A\sim_w SA, \quad \text{and}\quad SA\sim_w SB\;.
\end{equation}
For sequences $A=(a_n)_{n\in \mathbb{Z}}$ and $B=(b_n)_{n\in \mathbb{Z}}$ and constants $\alpha$ and $\beta$, we define the sequence $\alpha A+\beta B$ such that its $n$th entry is equal to $\alpha a_n+\beta b_n$ for each $n\in\mathbb{Z}$. We need the following result.
\begin{theorem}[{\cite[Theorem 2.7]{GMTW}}]
\label{thm:GMTW}
Let $A_1, \ldots, A_n$ be pairwise weakly synchronized sequences and let $u_1, \ldots, u_n, v_1, \ldots, v_n\geq 0$. Then, 
\[
\sum_{i=1}^{n}u_iA_i\sim_w \sum_{i=1}^{n}v_iA_i\;.
\]
\end{theorem}
As a consequence, we get the following induction hypothesis. 
\begin{lem}[Induction Hypothesis]
\label{lem:InductionHypothesis}
For every $m\geq 0$, let $A_m$ and $B_m$ denote the coefficient sequences of $F_{m}^{(1)}$ and $F_m^{(0)}$, respectively. Suppose that 
\begin{equation}
\label{eqn:Induction}
    A_m\sim_wB_m, \quad \text{and} \quad A_m\sim_w SB_m
\end{equation}
for some $m\geq 0$. Then, 
\[
A_{m+1}\sim_w B_{m+1}, \quad \text{and}\quad A_{m+1}\sim_w SB_{m+1}.
\]
\end{lem}
\begin{proof}
    From the assumptions~\eqref{eqn:Induction} and the fact that $B_m\sim_w SB_m$, we have that $A_m, B_m$ and $SB_m$ are pairwise weakly synchronized. It follows from Theorem~\ref{thm:GMTW} that 
    \[
    A_{m+1}=pA_m+0\cdot B_m+ qSB_m\sim_w pA_m+qB_m+0\cdot SB_m=B_{m+1}.
    \]
    Similarly, note that 
    \[
    A_{m+1} = pA_m+qSB_m, \qquad SB_{m+1} = pSA_m+qSB_m\;. 
    \]
    Therefore, $A_{m+1}\sim SB_{m+1}$ follows from Theorem~\ref{thm:GMTW} if $A_m, SA_m, BS_m$ are pairwise weakly synchronized. To see this, we note that $A_m\sim BS_m$ follows from the assumptions, while $A_m\sim_w SA_m$, $SA_m\sim_wSB_m$ follows from~\eqref{eqn:WeakSynchronizationFacts}, which completes the proof.
\end{proof}
We now have all the ingredients to prove Theorem~\ref{thm:logConcave}.
\begin{proof}[Proof of Theorem~\ref{thm:logConcave}]
For $p\in (0, 1)$ and $n\leq \nu$, the proof follows from Lemma~\ref{lem:logConcavitysmallpsmalln}. Now assume $p\in (0, 1)$ and $n>\nu$. By Lemma~\ref{lem:InductionHypothesis}, it suffices to prove~\eqref{eqn:Induction} for $m=0$ or $m=1$. In fact, when $r=\nu-1\neq 2$, we show that~\eqref{eqn:Induction} holds for $m=0$. When $r=\nu-1=2$, we will prove~\eqref{eqn:Induction} holds at $m=1$. 

\begin{enumerate}
    \item When $p\geq 1/2$, we note that $r=0$ giving $V_1(v)=V_0(v)=1$. In other words, $A_0=(1)$ and $B_0(1)$ and~\eqref{eqn:Induction} is trivial to verify.
    \item We now assume that $p\in (0, 1/2)$ equivalently $r\geq 1$. When $r=1$, we compute that $V_1(v)=(q+pv)$ and $v_0(v)=1$. In particular, $A_0=(q, p)$ and $B_0=(1)$ and~\eqref{eqn:Induction} is easily verified.
    \item We now consider the case $r=\nu-1=2$. In this case, we will verify~\eqref{eqn:Induction} at $m=1$. We begin by noting that $V_1(w)=(q+pv)^2$ and $V_0(v)=(1-p^2)+p^2v$. Therefore, 
   \begin{align*}
       F_1^{(1)}(v) &= pq^2+ q(1+p^2)v+ p^2v^2, \\
       F_1^{(0)}(v) &= q^2(1+2p)+3p^2qv+p^3v^2\;.
   \end{align*}
  The condition $A_1\sim_w B_1$ is equivalent to checking that 
   \[
   p^2q^2(1+2p)+p^4q^2 \leq 6p^2q^2(1+p^2),
   \]
   which simplifies to $(1+p)^2\leq 6(1+p^2)$. Since $p<1/2$, the last inequality is easily seen to be true. On the other hand, $A_1\sim_w SB_1$ is equivalent to checking the following two inequalities 
   \begin{align}
       (pq^2)(3p^2q) &\leq 2(q(1+p^2)\cdot q^2(1+2p)), \label{eqn:Firstcheck}\\
       q(1+p^2)\cdot p^3 &\leq 2 (p^3)\cdot (3p^2q) \label{eqn:SecondChcek}\;.
   \end{align}
   Note that~\eqref{eqn:Firstcheck} is equivalent to $3p^3\leq 2(1+p^2)(1+2p)$ which is immediate since $2(1+p^2)(1+2p)-3p^3 = 2+4p+2p^2+p^3>0$. On the other hand,~\eqref{eqn:SecondChcek} is equivalent to $1+p^2\leq 6p$, which is easy to verify since $p<1/2$.
   \item We now fix $r\geq 3$ and we will verify~\eqref{eqn:Induction} at $m=0$. Set $b_k= \binom{r}{k}p^kq^{r-k}$ for $1\leq k\leq r$ and zero otherwise. Note that
   \[
   A_0(k)= b_k,
   \]
   and 
   \[
   B_0(k) = \begin{cases}
   0, & k<0,\\
        b_0+b_1, & k=0, \\
       b_{k+1}, & 1\leq k\leq r-1.
   \end{cases}
   \]
   To verify $A_0\sim_w B_0$, we need to check that for all $k\in \mathbb{Z}$ we have
   \begin{equation}
   \label{eqn:FirstVerification}
       A_0(k-1)B_0(k+1)+A_0(k+1)B_0(k-1)\leq 2A_0(k)B_0(k).
   \end{equation}
   Since the left-hand side is zero if $k\leq 0$ or $k\geq r$, and therefore~\eqref{eqn:FirstVerification} holds trivially. We now verify~\eqref{eqn:FirstVerification} for $1\leq k\leq r-1$. We split the verification into two cases:

\noindent {\bf Case $k=1$.} Observe that 
       \[
       b_0b_3+b_2(b_0+b_1)\leq 2 b_1b_2\;.
       \]
       After some algebra, this inequality reduces to 
       \begin{equation}
       \label{eqn:ThetavsCr}
            \theta:=\frac{p}{q}\geq \frac{3}{2(r+1)}=:c_r\;.
       \end{equation}
       To show that $\theta\geq c_r$, we first show that $\theta(1+\theta)^r>1$ and $c_r(1+c_r)^r<1$ and use the elementary fact that $t\mapsto t(1+t)^r$ is an increasing function on $[0, \infty)$ for every $r\geq 3$. 

    We first show that $\theta(1+\theta)^r>1$. To this end, recall that $r=\nu-1$. By definition of $\nu$, we have that $ q^{r}\geq p> q^{r+1}$. Using the fact that $q=(1+\theta)^{-1}$ and $p=\theta(1+\theta)^{-1}$, and the fact that $p>q^{r+1}$, we conclude that $\theta(1+\theta)^r>1$. 

    To see that $c_r(1+c_r)^r<1$, we argue that $r\mapsto c_r(1+c_r)^r$ is decreasing for $r\geq 3$ and by a direct computation we evaluate that $c_3(1+c_3)^3<1$ which proves the claim. To prove that $c_r(1+c_r)^r$ is decreasing, define 
    \[
    H(r) :=\log c_r(1+c_r)^r = \log \frac{3}{2(r+1)}\left(1+\frac{3}{2(r+1)}\right)^r\;.
    \]
    We compute the derivative and using the fact that $\log(1+x)<x$, we obtain 
    \begin{align*}
        H'(r) &= -\frac{1}{r+1}+\log\left(1+\frac{3}{2(r+1)}\right) - \frac{3r}{(r+1)(2r+5)}\\
        &\leq  -\frac{1}{r+1}+\frac{3}{2(r+1)} - \frac{3r}{(r+1)(2r+5)}\\
        &= \frac{-4r+5}{2(r+1)(2r+5)}
    \end{align*}
    which is clearly negative when $r\geq 3$. This shows that $H(r)$ is decreasing and therefore so is $c_r(1+c_r)^r$. 

\noindent {\bf Case $2\leq k\leq r-1$.} In this case,~\eqref{eqn:FirstVerification} is equivalent to
       \[
       b_{k-1}b_{k+2}+b_{k+1}b_k \leq 2b_{k}b_{k+1},
       \]
       which is an immediate consequence of the log-concavity of the sequence $k\mapsto b_k$, which follows from Lemma~\ref{lem:binomlogconcave}.

We also need to verify that $A_{0}\sim_w SB_0$. This is equivalent to verifying:
\begin{equation}
\label{eqn:Secondverification}
    A_0(k-1)B_0(k)+A_0(k+1)B_0(k-2) \leq 2A_0(k)B_0(k-1),
\end{equation}
which is trivial except when $1\leq k\leq r-1$. We split this verification into several cases.

\noindent {\bf Case $k=1$.} Using the definitions, in this case, we check that~\eqref{eqn:Secondverification} reduces to 
\[
b_0b_2 \leq 2b_1(b_0+b_1)\;.
\]
Since Lemma~\ref{lem:binomlogconcave} gives $b_0b_2\leq b_1^2$, this is automatic. 

\noindent {\bf Case $k=2$.} In this case,~\eqref{eqn:Secondverification} reduces to 
\[
b_3(b_0+2b_1)\leq 2b_2^2\;.
\]
Note that Lemma~\ref{lem:binomlogconcave} only gives $b_1b_3\leq b_2^2$, which is not sufficient. We verify this by a direct computation. Indeed, using the definition of $b_k$ and expanding we get that $b_3(b_0+2b_1)\leq 2b_2^2$ is equivalent to 
\[
\binom{r}{3}q+2r\binom{r}{3}p \leq 2\binom{r}{2}^2p\;.
\]
Using the elementary inequalities
\[
2\binom{r}{2}^2 - 2r\binom{r}{3} = \frac{r^2(r-1)(r-2)}{6}, 
\]
we check that the above condition is equivalent to 
\[
\frac{p}{q} \geq \frac{r-2}{r(r+1)}\;.
\]
In~\eqref{eqn:ThetavsCr}, we showed that $\frac{p}{q}\geq \frac{3}{2(r+1)}$. The above inequality follows by noting that $\frac{3}{2(r+1)}\geq \frac{r-2}{r(r+1)}$.

\noindent {\bf Case $3\leq k\leq r-1$.} Equation~\eqref{eqn:Secondverification} in this regime is easily seen to be equivalent to 
\[
b_{k-1}b_{k+1}\leq b_k^2,
\]
which is precisely the log-concavity of $(b_k)_{k\in \mathbb{Z}}$ proved in Lemma~\ref{lem:binomlogconcave}.
  
\end{enumerate}

\end{proof}
\subsection{Central Limit Theorem and Local CLT}
In this subsection, we prove Theorem \ref{thm:CLT} and Theorem \ref{thm:LocalCLT} by doing complex analysis on the generating function of $X_n$. We will collect two results of Hwang which are found in \cite{FlaSed}. The first is Hwang's quasi-powers theorem.
\begin{proposition}[{\cite[Theorem IX.8]{FlaSed}}]
\label{prop:hwang_quasi_powers}
    Let the $X_n$ be non-negative discrete random variables with probability generating functions $s_n(v)$. Assume that,
uniformly in a fixed complex neighbourhood of $v = 1$, for sequences
$\beta_n, \kappa_n \to +\infty$, one has
\begin{equation*}
s_n(v)=  A(v)\, B(v)^{\beta_n} \left( 1 + O\!\left( \frac{1}{\kappa_n} \right) \right),
\end{equation*}
where $A(v)$, $B(v)$ are analytic at $v = 1$ and $A(1) = B(1) = 1$.
Assume finally that $B(v)$ satisfies the so-called variability condition,
\[
\fv(B(v))=B''(1) + B'(1) - B'(1)^2 \neq 0.
\]
Under these conditions, the distribution of $X_n$ is, after standardization, asymptotically Gaussian, and the
speed of convergence to the Gaussian limit is $O\left( \kappa_n^{-1} + \beta_n^{-1/2} \right)$, that is,
\begin{equation*}
\mathbb{P}\left(
\frac{X_n - \mathbb{E}(X_n)}{\sqrt{\mathbb{V}(X_n)}} \le x
\right)
=\Phi(x)+O\left(\frac{1}{\kappa_n}+\frac{1}{\sqrt{\beta_n}}\right),
\end{equation*}
where $\Phi(x)$ is the distribution function of a standard normal random variable,
\[
\Phi(x)=\frac{1}{\sqrt{2\pi}}
\int_{-\infty}^{x}
e^{-v^2/2}\, dv.
\]
\end{proposition}
Under additional assumptions on the probability generating functions as in Proposition \ref{prop:hwang_quasi_powers}, we can obtain a local limit law.
\begin{proposition}[{\cite[Theorem IX.14]{FlaSed}}]
\label{prop:localquasipowers}
    Let $X_n$ be a sequence of non-negative discrete random variables with probability generating functions $s_n(v)$. Assume that the functions $s_n(v)$ satisfy the conditions of Proposition \ref{prop:hwang_quasi_powers}, where, in particular, the quasi-power approximation
    \[s_n(v)=  A(v)\, B(v)^{\beta_n} \left( 1 + O\!\left( \frac{1}{\kappa_n} \right) \right)\]
    holds uniformly in a fixed complex neighbourhood $U$ of $1$. In addition, assume that there exists a uniform bound
    \[|s_n(v)|\leq K^{-\beta_n}\]
    for some $K>1$ and all $v$ in the intersection of the unit circle and the complement $\mathbb{C}\setminus  U$. Then, the distribution of $X_n$ satisfies a local limit law of the Gaussian type with speed of convergence $O\left(\beta_n^{-1/2}+\kappa_n^{-1}\right)$.
\end{proposition}
The outline of the remainder of this subsection is as follows. We first prove Theorem \ref{thm:CLT} and Theorem \ref{thm:LocalCLT} in the $p\geq 1/2$ case. Then we prove both theorems in the $p<1/2$ case.
\begin{proof}[Case $1$: $p\geq 1/2$.]
First, we prove Theorem \ref{thm:CLT}. We aim to apply Proposition \ref{prop:hwang_quasi_powers}. Our starting point is the generating function for $X_n$ from Proposition \ref{prop:asym-pgf}, that is,
\[s_n(v)=\frac{\bigl(s_2(v)-\lambda_-(v)s_1(v)\bigr)\lambda_+(v)^{\,n-1}
+\bigl(\lambda_+(v)s_1(v)-s_2(v)\bigr)\lambda_-(v)^{\,n-1}}
{\lambda_+(v)-\lambda_-(v)},
\] 
which holds for all $n\geq 1$ and all $p\geq 1/2$. Letting
\[A_1(v)=\frac{s_2(v)-\lambda_-(v)s_1(v)}{\lambda_+(v)-\lambda_-(v)}\]
and 
\[A_2(v)=\frac{\lambda_+(v)s_1(v)-s_2(v)}{\lambda_+(v)-\lambda_-(v)},\]
we observe that
\[s_n(v)=A_1(v)\lambda_+(v)^{n-1}\left(1+\left(\frac{A_2(v)}{A_1(v)}\right)\left(\frac{\lambda_-(v)}{\lambda_+(v)}\right)^{n-1}\right).\]

Since $\lambda_+(1)=1$ and $\lambda_-(1)=0$ and $\lambda_+$ and $\lambda_-$ are analytic, there exists a neighborhood $U$ of $1$ and a constant $\rho>1$ such that $|\lambda_-(v)/\lambda_+(v)|\leq1/\rho$ for all $v\in U$. A similar argument (since $A_1(1)=1$) shows that, shrinking $U$ if necessary, there exists a constant $C>0$ such that $|A_2(v)/A_1(v)|\leq C$ for all $v\in U$. It follows that
\begin{equation}
\label{eq:estimateonsn}
    s_n(v)=A_1(v)\lambda_+(v)^{n-1}\left(1+O(\rho^{-n})\right)
\end{equation}
uniformly for $v\in U$. 

Using the notation from Proposition \ref{prop:hwang_quasi_powers}, we check that
\begin{equation}
\label{eq:fvlambda+}
    \fv(\lambda_+)=\lambda_+''(1) + \lambda_+'(1) - \lambda_+'(1)^2=pq(1-3pq)\neq 0.
\end{equation}

Therefore, by Proposition \ref{prop:hwang_quasi_powers}, we have
\[\mathbb{P}\left(
\frac{X_n - \mathbb{E}(X_n)}{\sqrt{\mathbb{V}(X_n)}} \le x
\right)
=\Phi(x)+O\left(\frac{1}{\rho^n}+\frac{1}{\sqrt{n}}\right),
\]
where $\Phi(x)$ is the distribution function of a standard normal random variable. This implies Theorem \ref{thm:CLT} in the $p\geq 1/2$ case.

Next, we prove the local CLT in the $p\geq 1/2$ case. We want to apply Proposition \ref{prop:localquasipowers}. Set $\alpha=4pq$. As in Proposition \ref{prop:asym-pgf}, rewrite $\Delta(v)=v^2-\alpha v(v-1)=v((1-\alpha)v+\alpha)$. Assume that $|v|=1$. If $v\neq 1$, then $(1-\alpha)v+\alpha$ is a strict convex combination of the distinct points $v$ and $1$ on the unit circle $S^1$. Therefore, $|(1-\alpha)v+\alpha|<1$, so that $|\Delta(v)|<1$, which implies that $|\sqrt{\Delta(v)}|<1$ for every $v\in S^1\setminus\{1\}$. Combining this estimate with the triangle inequality implies that
\[|\lambda_{\pm}(v)|\leq\frac{|v|+|\sqrt{\Delta(v)}|}{2}<1\] 
for all $v\in S^1\setminus\{1\}$. With $U$ the open neighborhood of $1$ as before, by continuity and compactness, there exists a constant $K>1$ such that $\max\{|\lambda_-(v)|,|\lambda_+(v)|\}\leq 1/K$ for all $v\in S^1\setminus U$. The functions $A_1(v)$ and $A_2(v)$ extend continuously across any repeated-root point, so they are bounded on the compact set $S^1\setminus U$. Thus, there is a constant $M>0$ such that $|A_1(v)|+|A_2(v)|\leq M$ for all $v\in S^1\setminus U$. The decomposition
\[s_n(v)=A_1(v)\lambda_+(v)^{n-1}+A_2(v)\lambda_-(v)^{n-1}\]
implies that
\[|s_n(v)|\leq MK^{-(n-1)}\]
for all $v\in S^1\setminus U$. Combining this with the estimate in Equation \eqref{eq:estimateonsn} and Proposition \ref{prop:localquasipowers} yields the local CLT in Theorem \ref{thm:LocalCLT}.

\end{proof}
\begin{proof}[Case $2$: $p<1/2$]
    Again, we start with the appropriate generating function $s_n(v)$ for $X_n$. From Proposition \ref{prop:asym-pgf}, we have that for all large enough $n$ ($n\geq \nu$) and all $p\in(0,1/2)$,
        \[
            s_n(v)
            =
            \frac{\bigl(s_{\nu+1}(v)-\lambda_-(v)s_\nu(v)\bigr)\lambda_+(v)^{\,n-\nu}
            +\bigl(\lambda_+(v)s_\nu(v)-s_{\nu+1}(v)\bigr)\lambda_-(v)^{\,n-\nu}}
            {\lambda_+(v)-\lambda_-(v)},
        \]
    where $s_{\nu}$, $s_{\nu+1}$ and $\lambda_-$, $\lambda_+$ are defined in the proposition. Using a similar factorization as in the $p\geq 1/2$ case, we define
    \[A_1(v)=\frac{s_{\nu+1}(v)-\lambda_-(v)s_\nu(v)}{\lambda_+(v)-\lambda_-(v)}\]
    and
    \[A_2(v)=\frac{\lambda_-(v)s_\nu(v)-s_{\nu+1}(v)}{\lambda_+(v)-\lambda_-(v)}\]
    so that
    \[s_n(v)=A_1(v)\lambda_+(v)^{n-\nu}\left(1+\left(\frac{A_2(v)}{A_1(v)}\right)\left(\frac{\lambda_-(v)}{\lambda_+(v)}\right)^{n-\nu}\right).\]
    A similar argument as in the $p\geq 1/2$ case (noting $A_1$, $A_2$ are analytic and $A_2(1)=0$ and $A_1(1)=0$) implies that
    \begin{equation}
    \label{eq:estimatepgeq12}
        s_n(v)=A_1(v)\lambda_+(v)^{n-\nu}\left(1+O(1/\rho^{n-\nu})\right)
    \end{equation}
    for all $\rho$ in a fixed complex neighborhood $U$ of $1$. We computed that $\fv(\lambda_+)\neq 0$ in Equation $\eqref{eq:fvlambda+}$, so that by Proposition \ref{prop:hwang_quasi_powers},
    \begin{equation}
    \label{eq:convdistnufixed}   
        \mathbb{P}\left(
            \frac{X_n - \mathbb{E}(X_n)}{\sqrt{\mathbb{V}(X_n)}} \le x
            \right)
            =\Phi(x)+O\left(\frac{1}{\rho^{n-\nu}}+\frac{1}{\sqrt{n-\nu}}\right).
    \end{equation}
    Noting that $\nu$ is fixed, Equation $\eqref{eq:convdistnufixed}$ implies the CLT in the case when $p< 1/2$.

    Next, we prove the local CLT in the $p<1/2$ case. A similar argument from the $p\geq 1/2$ case implies that there is a constant $M>0$ and a constant $K>1$ (not necessarily the same constants as in the $p\geq 1/2$ case) such that
    \[|s_n(v)|\leq MK^{-(n-\nu)}\]
    uniformly on $S^1\setminus U$. Combining this estimate with the asymptotic in Equation $\eqref{eq:estimatepgeq12}$ and applying Proposition \ref{prop:localquasipowers} yields the local CLT when $p<1/2$.
\end{proof}
\subsection{Large Deviation Principle}
In this subsection, we establish the large deviation principle for the number of correct guesses $X_n$ in Theorem \ref{thm:LDP}. The proof is a direct application of the G\H{a}rtner--Ellis theorem~\cite[Theorem 2.3.6]{DZ10}. Identifying the rate function is the main technical part of the proof.
\begin{proof}
 Let us define the cumulant generating function for $X_n/n$:
\[
\Lambda_n(\theta) = \log \E\left[e^{\theta \frac{X_n}{n}}\right].
\]
And, let us define
    \[
    \Lambda_p(\theta):= \lim_{n\to \infty}\frac{1}{n}\Lambda_n(n\theta) = \lim_{n\to \infty}\frac{1}{n}\log\mathbb{E}[e^{\theta X_n}]\;.
    \]
Noting that $\E[e^{\theta X_n}]=s_n(e^{\theta})$ and using the estimates in~\eqref{eq:estimateonsn} and~\eqref{eq:estimatepgeq12}, we obtain
    \[\Lambda_p(\theta) = \log \lambda_{+}(e^{\theta})\;.\]
    It follows from the G\H{a}rtner--Ellis theorem (~\cite[Theorem 2.3.6]{DZ10}) that $X_n/n$ satisfies an LDP with speed $n$ and rate function given by the Legendre transform
    \[
    I_p(x):= \sup_{\theta\in \R}[\theta\cdot x - \Lambda_p(\theta)]\;.
    \]
We now evaluate the rate function $I_p$. To this end, we recall that 
    \[
    \lambda_{+}(e^{\theta}) = \frac{e^\theta}{2}(1+\sqrt{\tau+4\alpha e^{-\theta}}),
    \]
    where $\alpha=p(1-p)$ and $\tau=1-4\alpha=(2p-1)^2$. In particular, 
    \[
    \Lambda_p(\theta) = \theta-\log 2 +\log (1+\sqrt{\tau+4\alpha e^{-\theta}}).
    \]
    It is easy to check via a direct computation that $\Lambda_p'(\theta)\in (1/2, 1)$ and 
    \[
    \lim_{t\to -\infty}\Lambda_p'(t) = \frac{1}{2},\qquad \lim_{t\to +\infty}\Lambda_p'(t) = 1.
    \]
    In particular, if $x\not\in [1/2, 1]$, then $I_p(x)=+\infty$. Furthermore, notice that $\Lambda_p''(\theta)>0$; therefore, $\theta\mapsto \theta\cdot x-\Lambda_p(\theta)$ is strictly concave and has a unique maximum at 
   \[
    x = \Lambda_p'(\theta) = \frac{y^2+2y+\tau}{2y^2+2y},
    \]
    where $y=y_p(\theta)=\sqrt{\tau+4\alpha e^{-\theta}}$. Since $x\in (1/2, 1)$, we get that the unique solution is
    \[
    y=y_p(x) = \frac{(1-x)+\sqrt{(1-x)^2+\tau(2x-1)}}{(2x-1)}\;.
    \]
   Using the fact that $e^{-\theta} = \frac{y^2-\tau}{4\alpha}$, we write 
   \begin{align*}
      I_p(x) &= -x\log e^{-\theta} - \log\frac{e^{\theta}}{2} - \log (1+y)\\
      &= (1-x)\log \frac{y^2-\tau}{4\alpha} +\log 2 - \log (1+y)\;.
   \end{align*}
   It is easy to verify that $I_p$ is a good rate function. 
\end{proof}

\subsection{Phase Transitions}
In this subsection, we establish the phase transition in Theorem \ref{thm:PhaseTransition}.
\begin{proof}[Proof of Theorem \ref{thm:PhaseTransition}(1)]
    Let $p=1-\lambda/n^{\alpha}$. It suffices to show that $n-X_n\to 0$ in $L^1$. Taking expectations (and noting trivially that $n-X_n\geq 0$ a.s.), we obtain from Theorem \ref{thm:MeanAndVariance} that
    \begin{equation}
        \begin{aligned}
            \E|n-X_n|&=n(1-\lambda/n^{\alpha})(\lambda/n^{\alpha}))+2(1-\lambda/n^{\alpha})-2(1-\lambda/n^{\alpha})^2-\lambda/n^{\alpha}\\
            &=\lambda/n^{\alpha-1}+O(n^{-\alpha})
        \end{aligned}
    \end{equation}
    which tends to zero whenever $\alpha>1$. 
\end{proof}
\begin{remark}
    One can also see that $X_n=n$ if and only if the shuffle results in the the identity permutation $(1,2,\ldots,n)$ of cards, which happens with probability $p^{n-1}=(1-\lambda/n^{\alpha})^{n-1}$. This probability tends to $1$ when $n\to\infty$ if $\alpha>1$, giving another proof of Theorem \ref{thm:PhaseTransition}.
\end{remark}
Recall that a random variable $X$ is $\text{Poisson}(\lambda)$-distributed if $\mathbb{P}(X=k)=e^{-\lambda}\lambda^k/k!$ for all integers $k\geq 0$. A $\text{Poisson}(\lambda)$ limiting distribution of the number of correct guesses when $p=1-\lambda/n$ can be deduced directly from the probability mass function of $X_n$. 
\begin{proof}[Proof of Theorem \ref{thm:PhaseTransition}(2)]
    Let $p=1-\lambda/n$, and let $q=1-p$. From replacing $k$ with $n-k$ in Theorem \ref{thm:DistributionXn}, we obtain that
    \[\mathbb{P}(n-X_n=k)=\sum_{m=k}^{n-k}\binom{m-1}{k-1}\binom{n-m}{k}q^{m}p^{n-1-m}\]
    for $1\leq k\leq n/2$. We will take the limit as $n\to\infty$ of this sum and match it with the mass function of the Poisson distribution. First, factor out the weights involving $p$ and $q$, and reindex by letting $j=m-k$. We obtain
    \[\mathbb{P}(n-X_n=k)=q^kp^{n-k-1}\sum_{j=0}^{n-2k}\binom{j+k-1}{k-1}\binom{n-k-j}{k}(q/p)^j.\]
    Note that $q^kp^{n-k-1}\sim e^{-\lambda} \lambda^k/n^k$. Now, let
    \[M_{n,k}(x)=\sum_{j=0}^{n-2k}\binom{j+k-1}{k-1}\binom{n-k-j}{k}x^j.\]
    Then note the trivial lower bound by the $j=0$ term, that is,
    \[M_{n,k}(x)\geq \binom{n-k}{k}.\]
    For an upper bound, note that for all $j\geq 0$, we have $\binom{n-k-j}{k}\leq \binom{n}{k}$, so that
    \begin{equation} 
        \begin{aligned}
            M_{n,k}(x)&\leq\binom{n}{k}\sum_{j=0}^{n-2k}\binom{j+k-1}{k-1}x^j\\
            &\leq\binom{n}{k}\sum_{j=0}^{\infty}\binom{j+k-1}{k-1}x^j \\
            &=\binom{n}{k}\frac{1}{(1-x)^k}.
        \end{aligned}
    \end{equation}
    Furthermore, we have $x=q/p=\lambda/(n-\lambda)\to 0$ so that $(1-x)^{-k}\sim 1$. It follows that
    \[\binom{n-k}{k}\leq M_{n,k}(x)\lesssim\binom{n}{k}\]
    so that
    \begin{equation}
    \label{eq:poissonexpressionXn}
        \frac{e^{-\lambda} \lambda^k}{n^k}\binom{n-k}{k}\lesssim\mathbb{P}(n-X_n=k)\lesssim \frac{e^{-\lambda} \lambda^k}{n^k}\binom{n}{k}.
    \end{equation}
    Note that both $\binom{n}{k}\sim n^k/k!$ and $\binom{n-k}{k}\sim n^k/k!$, and observe that $\mathbb{P}(n-X_n=0)=\mathbb{P}(X_n=n)=p^{n-1}\to e^{-\lambda}$. The result is then immediate after passing to the limit in Equation $\eqref{eq:poissonexpressionXn}$.
\end{proof}
Theorem \ref{thm:PhaseTransition} can be deduced from the log-concavity of $X_n$ demonstrated in Theorem \ref{thm:logConcave}. 
\begin{proof}[Proof of Theorem \ref{thm:PhaseTransition}(3)]
Taking the asymptotics of the variance in Theorem \ref{thm:MeanAndVariance} and substituting $p=1-\lambda/n^{\alpha}$, we have
\begin{equation} 
    \begin{aligned}
        \Var(X_n)&=(1-p)p(3p^2-3p+1)n+p(1-p)(10p-8p^2-3)\\                &=\lambda n^{1-\alpha}
-\lambda n^{-\alpha}
-4\lambda^2 n^{1-2\alpha}
+7\lambda^2 n^{-2\alpha}
+6\lambda^3 n^{1-3\alpha}\\
&\;\;\;\;\;\;\;\;\;\;\;\;\;-14\lambda^3 n^{-3\alpha}
-3\lambda^4 n^{1-4\alpha}
+8\lambda^4 n^{-4\alpha}.\\      
    \end{aligned}
\end{equation}
where a computer algebra system was used to simplify the expression. The first term $\lambda n^{1-\alpha}$ gives the leading order asymptotics, so that $\Var(X_n)\to\infty$ when $\alpha\in(0,1)$. Since $X_n$ is log-concave by Theorem \ref{thm:logConcave}, Harper's method \cite{Pitman} implies that, centered by its mean and scaled by its standard deviation, $X_n$ is asymptotically normal when $\Var(X_n)\to\infty$, i.e. when $\alpha\in(0,1)$.

\end{proof}
Finally, we address the phase transition near $p=0$. We recall that the Lambert $W$-function $W(x)$ \cite{LambertWFunction} is defined as the inverse of the mapping $x\mapsto xe^x$.
\begin{proof}[Proof of Theorem \ref{thm:phasetransitionp0}]
        Directly from the definition of $\nu$ in Equation $\eqref{eq:defofnu}$, we have that $n>\nu$ (resp. $n\leq\nu$) when $p$ is chosen such that $p>(1-p)^{n-1}$ (resp. $p\leq (1-p)^{n-1}$). Since we are assuming that $p\to 0$, it suffices to analyze the equation $p=(1-p)^n$ when $n$ is large. Let $p_n\in(0,1)$ be the unique sequence which satisfies the equation $p_n=(1-p_n)^n$ for all $n\geq 1$. Then, as $n\to\infty$, we have
        \begin{equation*}
        \begin{aligned}
            p_n&=(1-p_n)^n\\
            &=e^{n\log(1-p_n)}\\
            &\sim e^{-np_n+O(np_n^2)}\\
        \end{aligned}
        \end{equation*}
        so that, multiplying both sides by $ne^{np_n}$, we have
        \begin{equation}
        \label{eq:pnasymptotic}
            np_ne^{np_n}\sim ne^{O(np_n^2)}.
        \end{equation}
        If we assume that $p_n\ll n^{-1/2}$ and take $x=np_n$, we find that Equation $\eqref{eq:pnasymptotic}$ is satisfied if and only if $xe^x\sim n$, i.e. $x\sim W(n)$, where $W$ is the Lambert $W$-function. Therefore, we have the asymptotic $p_n\sim W(n)/n$. Now, the function $W(n)$ has the asymptotic
        \[W(n)\sim\ln n-\ln\ln n+o(1)\]
        given by \cite[Equation (4.18)]{LambertWFunction}. We deduce that when
        \[p\ll p_n\sim\frac{\ln n-\ln\ln n}{n},\]
        we have $n\leq\nu$ for all large $n$, and when $p\gg p_n$, we have that $n>\nu$ for all large $n$.

        For all $n\geq 1$, we let $Y_n$ have the probability mass function given by Theorem $\ref{thm:DistributionXn}(2)$, and let $Z_n$ have the probability mass function given by Theorem $\ref{thm:DistributionXn}(3)$. By Theorem $\ref{thm:DistributionXn}$, for any $0<p<1/2$ and all $n\geq 1$, we have
        \begin{equation}
        \label{eq:decompXnindicators}
            X_n\law Y_n\mathbb{I}(n\leq \nu)+Z_n\mathbb{I}(n> \nu).
        \end{equation}
        Slutsky's theorem implies that $X_n$ has the same limiting behavior as $Y_n$ (resp. $Z_n$) when $p$ is chosen such that $n\leq\nu$ (resp. $n>\nu$) for all large $n$. 
        
        A quick glance at the mass function of $Y_n$ in Theorem $\ref{thm:DistributionXn}(2)$ gives that $Y_n \law 1+\operatorname{Bin}(n-1,q)$, so that $\Var(Y_n)\to\infty$ if $p\gg 1/n$. The classical central limit theorem and Slutsky's theorem applied to $Y_n$ yields Theorem $\ref{thm:phasetransitionp0}(1)$. On the other hand, notice that $\P(Y_n=n)=q^{n-1}\to 1$ when $p\ll 1/n$, which yields Theorem $\ref{thm:phasetransitionp0}.2$. 

        If $p\gg (\ln n-\ln \ln n)/n$, we can appeal to the log-concavity of $Z_n$ established in Theorem $\ref{thm:logConcave}$. Harper's method \cite{Pitman} implies that $Z_n$, centered by its mean and variance, converges to a standard normal random variable when $\Var(Z_n)\to\infty$. 
        Before analyzing $\Var(Z_n)$, we note two facts. First,
        \begin{equation}
        \label{eq:asymptoticforq^nu}
            q^{\nu}=q^{\lfloor\ln(p)/\ln(q)\rfloor+1}\sim q^{\ln(p)/\ln(q)+1}=pq.
        \end{equation}
        Second, using a Taylor expansion for $\ln(1-p)$, we have
        \begin{equation}
        \label{eq:asymptoticfornu}
            \nu\sim \frac{|\ln(p)|}{|\ln(1-p)|}\sim\frac{|\ln(p)|}{p+\frac{p^2}{2}+O(p^3)}\sim\frac{|\ln(p)|}{p}+o(1).
        \end{equation}
        Hence, we have that for all $n\geq \nu+2$ and any sequence $p\to 0$,
        \begin{equation}
        \label{eq:asymptoticvarzn}
        \begin{aligned}
            \Var(Z_n)&=npq(1-3pq)-2pq+(3\nu+5)p^2q^2+\bigl(1-2\nu p+2p^2\bigr)q^\nu-q^{2\nu}\\
            &= npq(1-3pq)-2pq+(3\nu+5)p^2q^2+\bigl(1-2\nu p+2p^2\bigr)pq-p^2q^{2}+o(1)\\
            &=npq(1-3pq)+\nu (3p^2q^2-2p^2 q) +o(1)\\
            &= npq(1-3pq)+p^{-1}|\ln(p)|(3p^2q^2-2p^2q)+o(1)\\
            &=npq(1-3pq)+p|\ln(p)|(3q^2-2q)+o(1)\\
            &= npq(1-3pq)-2p|\ln(p)|+o(1)\\
            &=npq(1-3pq)+o(1).
        \end{aligned}
        \end{equation}
        The first equality used Theorem $\ref{thm:MeanAndVariance}$, the second and fourth used Equations $\eqref{eq:asymptoticforq^nu}$ and $\eqref{eq:asymptoticfornu}$, respectively, and the last used the fact that $p|\ln(p)|\to 0$ as $p\to 0$. It follows that $\Var(Z_n)\to\infty$ for all $p\gg 1/n$. In particular, this means that $\Var(Z_n)\to\infty$ when $p\gg (\ln n-\ln\ln n)/n$, the threshold at which $Z_n$ and $X_n$ begin to share the same limiting behavior for all large $n$. This concludes the proof.
yield\end{proof}

\section{Discussion and Open Questions}
\subsection{Complete feedback game} The asymmetric single-shelf shuffle can be generalized to an asymmetric $(m,p)$-shelf shuffle as done in Tripathi \cite{Tri26cut} \cite{tri26analysis}. Setting $p=1/2$ in this model recovers the $m$-shelf shuffle introduced in \cite{DiaconisFulmanHolmes2013}, and setting $m=1$ and fixing $p$ recovers the asymmetric single shelf shuffle from this paper. In \cite[Section 5.1]{DiaconisFulmanHolmes2013}, Diaconis, Fulman, and Holmes conjectured an optimal guessing strategy with complete feedback for the symmetric $m$-shelf shuffle. Clay \cite{Clay2025a} proved their conjecture in the $m=1$ case, which is recovered by taking $p=1/2$ in Proposition $\ref{prop:Strategy}$. He added \cite[Conjecture 3]{Clay2025a}, without proof, that when the ratio of cards to shelves $n/m$ is ``not too large", the expected reward under their strategy should be approximately $nH_{2m}/2m$, where $H_{2m}=1+1/2+1/3+\ldots+1/2m$ is the $2m$-th harmonic number. This agrees with Monte-Carlo simulations in \cite{DiaconisFulmanHolmes2013}. Complete-feedback card guessing for the asymmetric $(m,p)$-shelf shuffle has not yet been analyzed rigorously for $m\geq 2$, even in the case when $p=1/2$.

\subsection{No feedback game}
In a no-feedback game, the player guesses the cards sequentially from a shuffled deck. After each guess, the card is removed from the deck. In this setup, however, the player gets no further information. The goal of the player is to maximize the expected number of correct guesses. 
In this case, it is easy to describe the optimal strategy abstractly. At location $j$, the player should guess a card $i$ that has the maximum chance of appearing at location $j$. More precisely, an optimal guessing strategy will be to guess the label $g_j$ at location $j$ such that 
\[
g_j\in \arg\max\{m_{i, j}: 1\leq i\leq n\}\;,
\]
where $m$ is the position matrix in Proposition~\ref{prop:posi}.
An optimal strategy is generally not unique. However, for any optimal strategy, the expected number of correct guesses is 
\[
S_n := \sum_{j=1}^{n} \max_{1\leq i\leq n} m_{i, j}\;.
\]

The no-feedback case seems more challenging, and there are only partial results known even for $p=1/2$. For the symmetric single-shelf shuffle, the no feedback game has been analyzed in~\cite{Clay2025a, Tripathi2026}. However, the exact determination of an optimal strategy remains elusive even in this case. Only the asymptotics of $S_n$ are known when $p=1/2$. Following the same idea as in~\cite{Clay2025a, Tripathi2026}, we can obtain asymptotics for the expected number of correct guesses in no feedback game under the general asymmetric single-shelf shuffle. 
\begin{theorem}
\label{thm:ScoreNoFeedback}
Fix $p\in (0, 1)$. Let $q:=1-p$ and let $S_n$ be the expected number of correct guesses under an optimal straegy in no feedback game after an asymmetric single-shelf shuffle. Then, there exists a constant $C=C_p>0$ depending only on $p$ such that 
\[
\left|S_n - A(p)\sqrt{n}\right|\leq C\; \quad \text{for all } n\geq 1,
\]
where 
\[
 A(p)=\sqrt{\frac{2}{\pi}}
 \left(
 \frac{p^{3/2}}{\sqrt q}+
 \frac{q^{3/2}}{\sqrt p}
 \right).
\]
\end{theorem}
When $p=1/2$, we get $A(p)=\sqrt{\frac{2}{\pi}}$ and this recovers~\cite[Theorem 1.3]{Clay2025a} and~\cite[Theorem II]{Tripathi2026}.
\begin{proof}
For $t\ge 0$, let
\[
 b_{a,t}(r):=\binom{t}{r}a^r(1-a)^{t-r},
 \qquad 0\le r\le t,
\]
and set
\[
 M_{a,n}(r):=\max_{0\le t\le n} b_{a,t}(r).
\]
Notice that
\[f_{a,t}(r):=\frac{b_{a,t+1}(r)}{b_{a,t}(r)}=\frac{t+1}{t+1-r}(1-a)\]
so that $f_{a,t}(r)>1$ for $t<-1+r/a$ and $f_{a,t}(r)<1$ for $t>-1+r/a$. It follows that $b_{a,t}(r)$ is maximized when $t=\min(n,\lfloor r/a\rfloor)$. Therefore, 
\[M_{a,n}(r)=\begin{cases}
    \binom{\lfloor r/a\rfloor}{r}a^r(1-a)^{\lfloor r/a\rfloor-r}, & 0\leq r\leq\lfloor an\rfloor \\
    \binom{n}{r}a^r(1-a)^{n-r}, & \lfloor an\rfloor+1\leq r\leq n.
\end{cases}\]
Notice that if $X$ is a binomial random variable with parameters $(\lfloor r/a\rfloor,a)$ then $M_{a,n}(r)=\mathbb{P}(X=r)=\mathbb{P}(X=\E[X])$ for all $0\leq r\leq\lfloor an\rfloor$. It follows from Stirling's estimate that
\begin{equation}
\label{eq:firstboundonman}
    M_{a,n}(r)\sim\frac{1}{\sqrt{2\pi(1-a)r}}
\end{equation}
for all $0\leq r\leq\lfloor an\rfloor$. On the other hand, we have that if $Y$ is a binomial random variable with parameters $(n,a)$, then for all $\lfloor an\rfloor+1\leq r\leq n$, we have that $M_{a,n}(r)=\mathbb{P}(Y=r)$. Hence, we can apply Stirling's estimate to the mass function of $Y$ when $|r-an|=O(1)$ and get that the bound in Equation $\eqref{eq:firstboundonman}$ holds for all $0\leq r\leq \lfloor an\rfloor+O(1)$. It follows that
\begin{equation}
\begin{aligned}
    \sum_{r=0}^{\lfloor an\rfloor+O(1)} M_{a,n}(r)&\sim\sum_{r=0}^{\lfloor an\rfloor+O(1)} \frac{1}{\sqrt{2\pi(1-a)r}}\\
    &=\sqrt{\frac{2a}{\pi(1-a)}}\sqrt{n}+O_a(1),\label{eqn:Truncated}\\
\end{aligned}
\end{equation}
This implies that
\begin{equation}
\begin{aligned}
    \sum_{r=0}^n M_{a,n}(r)&=\sqrt{\frac{2a}{\pi(1-a)}}\sqrt{n}+O_a(1)+\sum_{r=\lfloor an\rfloor+O(1)}^n M_{a,n}(r)\\
    &=\sqrt{\frac{2a}{\pi(1-a)}}\sqrt{n}+O_a(1)+\sum_{r=\lfloor an\rfloor+O(1)}^n \mathbb{P}(Y=r)\\
    &=\sqrt{\frac{2a}{\pi(1-a)}}\sqrt{n}+O_a(1)\label{eqn:full}
\end{aligned}
\end{equation}
where in passing from the second to the third equality, the $O_a(1)$ term has increased by at most $1$ since $\mathbb{P}(Y=r)$ is a probability measure on $0\leq r\leq n$ as the parameters of $Y$ do not depend on $r$. 

To finish the proof, we first note from Proposition $\ref{prop:posi}$ that 
\[
 m_{ij} =p\,b_{i-1,p}(j-1)+q\,b_{i-1,q}(n-j), \qquad 1\le i,j\le n.
\]
To get the bound on $S_n$, we use~\eqref{eqn:full} to get
\begin{align*}
 S_n &\le p\sum_{r=0}^{n-1} M_{p,n}(r)+q\sum_{r=0}^{n-1} M_{q,n}(r)  \\
 &=\sqrt{\frac{2}{\pi}}
 \left(\frac{p^{3/2}}{\sqrt q}+\frac{q^{3/2}}{\sqrt p}\right)
 \sqrt{n}+O_p(1).
\end{align*}
For the lower bound, we split the sum at $R=\lfloor pn\rfloor$.
When
$0\le r\le R$, use the first summand in $m_{ij}$, and for
$R+1\le r\le n-1$, use the second. This gives

\begin{align*}
 S_n&\ge p\sum_{r=0}^{R} M_{p,n}(r)
   +q\sum_{r=R+1}^{n-1}M_{q,n}(N-r)  \\
 &=p\sum_{r=0}^{\lfloor pn\rfloor} M_{p,n}(r)
   +q\sum_{\ell=0}^{n-R-1}M_{q,n}(\ell).
\end{align*}

Note that $n-R-1=qn+O(1)$. We use~\eqref{eqn:Truncated} to conclude 
\[
S_n
 \ge
 \sqrt{\frac{2}{\pi}}
 \left(\frac{p^{3/2}}{\sqrt q}+\frac{q^{3/2}}{\sqrt p}\right)
 \sqrt{n}+O_p(1).
\]
\end{proof}

Let us remark that the above proof already provides a strategy that is approximately optimal, in the sense that for the first $\lfloor pn\rfloor$ cards, the player guesses the card $\lfloor r/p\rfloor$ and on the remaining cards the player guesses the $r$-th card at the $r$-th position. This is precisely how we obtain the lower bound on $S_n$. Determining an exact optimal strategy is more delicate even for $p=1/2$ (see the Discussion in~\cite{Tripathi2026}). A more tractable and practically relevant question is to understand the number of correct guesses one can make, under no feedback, after repeated shelf shuffles. When $p=1/2$, Chen and Ottolini~\cite{ChenOttolini2025} proved cutoff for the repeated symmetric single-shelf shuffle and showed that in roughly a $\frac{5}{4}\log n$ repetitions the shuffled deck is sufficiently close to being uniform. For a uniformly random deck, it is easy to see that under any strategy, with no feedback, the expected number of correct guesses is at most $1$. Let $S_n^{(k)}$ denote the number of correct guesses, with no feedback, after a symmetric single-shelf shuffle is repeated $k$ times. Tripathi~\cite{Tripathi2026} showed that $|S_n^{(k)}-1|\leq Cn^{-2\epsilon}$ for $k>(1+\epsilon)\log n$. A similar analysis for the riffle-shuffle was done by Ciucu~\cite{Ciucu98Nofeedback}. It would be interesting to extend this analysis to the asymmetric shelf shuffle.

\section*{Declarations of interest and Acknowledgments}
There are no competing financial or personal interests that influenced the work reported in this paper. A.C. would like to thank his advisor, Jason Fulman, for helpful conversations.

\bibliographystyle{alpha}
\bibliography{ref}{}

@article {AldousDiaconis1986,
    AUTHOR = {Aldous, David and Diaconis, Persi},
     TITLE = {Shuffling cards and stopping times},
   JOURNAL = {Amer. Math. Monthly},
  FJOURNAL = {American Mathematical Monthly},
    VOLUME = {93},
      YEAR = {1986},
    NUMBER = {5},
     PAGES = {333--348},
      ISSN = {0002-9890,1930-0972},
   MRCLASS = {60C05 (60B15 60G40 60J15)},
  MRNUMBER = {841111},
MRREVIEWER = {Endre\ Cs\'aki},
       DOI = {10.2307/2323590},
       URL = {https://doi.org/10.2307/2323590},
}

@article {BlackwellHodges1957,
    AUTHOR = {Blackwell, David and Hodges, Jr., J. L.},
     TITLE = {Design for the control of selection bias},
   JOURNAL = {Ann. Math. Statist.},
  FJOURNAL = {Annals of Mathematical Statistics},
    VOLUME = {28},
      YEAR = {1957},
     PAGES = {449--460},
      ISSN = {0003-4851},
   MRCLASS = {62.0X},
  MRNUMBER = {88849},
MRREVIEWER = {D.\ G.\ Chapman},
       DOI = {10.1214/aoms/1177706973},
       URL = {https://doi.org/10.1214/aoms/1177706973},
}

@article {Proschan91,
    AUTHOR = {Proschan, Michael},
     TITLE = {A note on {D}. {H}. {B}lackwell and {J}. {L}. {H}odges, {J}r.:
              ``{D}esign for the control of selection bias'' [{A}nn.\
              {M}ath.\ {S}tatist.\ {\bf 28} (1957), 449--460; {MR} {\bf 19},
              589], and {P}. {D}iaconis and {R}. {L}. {G}raham: ``{T}he
              analysis of sequential experiments with feedback to subjects''
              [{A}nn.\ {S}tatist.\ {\bf 9} (1981), no.\ 1, 3--23;
              {MR}0600529 (82f:62153)]},
   JOURNAL = {Ann. Statist.},
  FJOURNAL = {The Annals of Statistics},
    VOLUME = {19},
      YEAR = {1991},
    NUMBER = {2},
     PAGES = {1106--1108},
      ISSN = {0090-5364,2168-8966},
   MRCLASS = {60J15 (60C05 62E20 62P10)},
  MRNUMBER = {1105868},
MRREVIEWER = {James\ Allen\ Fill},
       DOI = {10.1214/aos/1176348144},
       URL = {https://doi.org/10.1214/aos/1176348144},
}

@article {ChenOttolini2025,
    AUTHOR = {Chen, Ray and Ottolini, Andrea},
     TITLE = {Cutoff in total variation for the shelf shuffle},
   JOURNAL = {Electron. Commun. Probab.},
  FJOURNAL = {Electronic Communications in Probability},
    VOLUME = {30},
      YEAR = {2025},
     PAGES = {Paper No. 44, 10},
      ISSN = {1083-589X},
   MRCLASS = {60J10},
  MRNUMBER = {4908782},
       DOI = {10.1214/25-ecp691},
       URL = {https://doi.org/10.1214/25-ecp691},
}

@article {diaconis1996cutoff,
    AUTHOR = {Diaconis, Persi},
     TITLE = {The cutoff phenomenon in finite {M}arkov chains},
   JOURNAL = {Proc. Nat. Acad. Sci. U.S.A.},
  FJOURNAL = {Proceedings of the National Academy of Sciences of the United States of America},
    VOLUME = {93},
      YEAR = {1996},
    NUMBER = {4},
     PAGES = {1659--1664},
      ISSN = {0027-8424},
   MRCLASS = {60J10 (60J15)},
  MRNUMBER = {1374011},
MRREVIEWER = {James\ Allen\ Fill},
       DOI = {10.1073/pnas.93.4.1659},
       URL = {https://doi.org/10.1073/pnas.93.4.1659},
}

@article{Clay2025a,
  author        = {A.~Clay},
  title         = {Guessing Strategies for Shuffling Machines},
  journal       = {arXiv preprint},
  year          = {2025},
  eprint        = {2507.10294},
  archivePrefix = {arXiv},
  primaryClass  = {math.PR}
}

@article{Clay2025b,
  author        = {A.~Clay},
  title         = {Limit Theorems for Descents and Inversions of Shelf-Shuffles},
  journal       = {arXiv preprint},
  year          = {2025},
  eprint        = {2510.00343},
  archivePrefix = {arXiv},
  primaryClass  = {math.PR}
}

@article {DiaconisFulmanHolmes2013,
    AUTHOR = {Diaconis, Persi and Fulman, Jason and Holmes, Susan},
     TITLE = {Analysis of casino shelf shuffling machines},
   JOURNAL = {Ann. Appl. Probab.},
  FJOURNAL = {The Annals of Applied Probability},
    VOLUME = {23},
      YEAR = {2013},
    NUMBER = {4},
     PAGES = {1692--1720},
      ISSN = {1050-5164,2168-8737},
   MRCLASS = {60C05 (05A15)},
  MRNUMBER = {3098446},
MRREVIEWER = {Kent\ E.\ Morrison},
       DOI = {10.1214/12-aap884},
       URL = {https://doi.org/10.1214/12-aap884},
}

@book {DiaconisFulman2023Shuffling,
    AUTHOR = {Diaconis, Persi and Fulman, Jason},
     TITLE = {The mathematics of shuffling cards},
 PUBLISHER = {American Mathematical Society, Providence, RI},
      YEAR = {[2023] \copyright 2023},
     PAGES = {xii+346},
      ISBN = {[9781470463038]},
   MRCLASS = {60-02 (05Axx 60B15 60J10 91A60)},
  MRNUMBER = {4565368},
MRREVIEWER = {Martin\ V.\ Hildebrand},
}

@article {DiaconisGraham1981,
    AUTHOR = {Diaconis, Persi and Graham, Ronald},
     TITLE = {The analysis of sequential experiments with feedback to
              subjects},
   JOURNAL = {Ann. Statist.},
  FJOURNAL = {The Annals of Statistics},
    VOLUME = {9},
      YEAR = {1981},
    NUMBER = {1},
     PAGES = {3--23},
      ISSN = {0090-5364,2168-8966},
   MRCLASS = {62L99 (60G40 92A25)},
  MRNUMBER = {600529},
MRREVIEWER = {Colin\ L.\ Mallows},
       URL =
              {http://links.jstor.org/sici?sici=0090-5364(198101)9:1<3:TAOSEW>2.0.CO;2-E&origin=MSN},
}

@article{Fisher1936,
author  = {R.~A.~Fisher},
    title   = {Design of experiments},
    journal = {British Medical Journal},
    year    = {1936},
    volume ={1},
    number ={3923},
		pages={554},
 }

@article{Diaconis1978,
		author  = {P.~Diaconis},
    title   = {Statistical problems in ESP research},
    journal = {Science},
    year    = {1978},
    volume ={201},
    number ={4351},
		pages={131--136},
 }

@article {Diaconis2022a,
    AUTHOR = {Diaconis, Persi and Graham, Ron and He, Xiaoyu and Spiro, Sam},
     TITLE = {Card guessing with partial feedback},
   JOURNAL = {Combin. Probab. Comput.},
  FJOURNAL = {Combinatorics, Probability and Computing},
    VOLUME = {31},
      YEAR = {2022},
    NUMBER = {1},
     PAGES = {1--20},
      ISSN = {0963-5483,1469-2163},
   MRCLASS = {60C05},
  MRNUMBER = {4356453},
MRREVIEWER = {Toshio\ Nakata},
       DOI = {10.1017/s0963548321000134},
       URL = {https://doi.org/10.1017/s0963548321000134},
}

@article {Efron1971,
    AUTHOR = {Efron, Bradley},
     TITLE = {Forcing a sequential experiment to be balanced},
   JOURNAL = {Biometrika},
  FJOURNAL = {Biometrika},
    VOLUME = {58},
      YEAR = {1971},
     PAGES = {403--417},
      ISSN = {0006-3444,1464-3510},
   MRCLASS = {62L05},
  MRNUMBER = {312660},
MRREVIEWER = {S.\ Blumenthal},
       DOI = {10.1093/biomet/58.3.403},
       URL = {https://doi.org/10.1093/biomet/58.3.403},
}

@article {HeOttolini2021,
    AUTHOR = {He, Jimmy and Ottolini, Andrea},
     TITLE = {Card guessing and the birthday problem for sampling without
              replacement},
   JOURNAL = {Ann. Appl. Probab.},
  FJOURNAL = {The Annals of Applied Probability},
    VOLUME = {33},
      YEAR = {2023},
    NUMBER = {6B},
     PAGES = {5208--5232},
      ISSN = {1050-5164,2168-8737},
   MRCLASS = {60C05},
  MRNUMBER = {4677732},
MRREVIEWER = {Lyuben\ R.\ Mutafchiev},
       DOI = {10.1214/23-aap1946},
       URL = {https://doi.org/10.1214/23-aap1946},
}

@unpublished{Gilbert1955,
    author  = {E.~Gilbert},
    title   = {Theory of shuffling},
     year    = {1955},
note = {Technical memorandum, Bell Labs},  
 }

@book {Kraitchik1942Recreations,
    AUTHOR = {Kraitchik, Maurice},
     TITLE = {Mathematical recreations},
      NOTE = {2d ed},
 PUBLISHER = {Dover Publications, Inc., New York},
      YEAR = {1953},
     PAGES = {330},
   MRCLASS = {10.0X},
  MRNUMBER = {52446},
}

@article{Markov,
author  = {A.~Markov},
    title   = {A Extension of the law of large numbers to dependent events (Russian).},
    journal = {Bull. Soc. Math. Kazan},
    year    = {1906},
volume={2},
  pages={155--156},
 }

@article {OttoliniSteiner2023,
    AUTHOR = {Ottolini, Andrea and Steinerberger, Stefan},
     TITLE = {Guessing cards with complete feedback},
   JOURNAL = {Adv. in Appl. Math.},
  FJOURNAL = {Advances in Applied Mathematics},
    VOLUME = {150},
      YEAR = {2023},
     PAGES = {Paper No. 102569, 16},
      ISSN = {0196-8858,1090-2074},
   MRCLASS = {62L99 (60G40)},
  MRNUMBER = {4604799},
       DOI = {10.1016/j.aam.2023.102569},
       URL = {https://doi.org/10.1016/j.aam.2023.102569},
}

@book{Poincare,
author  = {H.~Poincare},
    title   = {Calcul des probabilit\'es},
   publisher = {Gauthier Villars, Paris},
edition ={2nd},
    year    = {1912},
 }

@article {Sellke,
    AUTHOR = {Sellke, Mark},
     TITLE = {Cutoff for the asymmetric riffle shuffle},
   JOURNAL = {Ann. Probab.},
  FJOURNAL = {The Annals of Probability},
    VOLUME = {50},
      YEAR = {2022},
    NUMBER = {6},
     PAGES = {2244--2287},
      ISSN = {0091-1798,2168-894X},
   MRCLASS = {60C05 (20P05 60B15 60J10)},
  MRNUMBER = {4499839},
MRREVIEWER = {Wojciech\ Bartoszek},
       DOI = {10.1214/22-aop1582},
       URL = {https://doi.org/10.1214/22-aop1582},
}

@article{KP2024,
    author  = {M.~Kuba and A.~Panholzer},
    title   = {On Card guessing with two types of cards},
		journal = {J. Statist. Plann. Inference},
		volume  = {232},
    year    = {2024},
		pages   = {Paper No. 106160, 18~pages},
 }

@article {OT2024,
    AUTHOR = {Ottolini, Andrea and Tripathi, Raghavendra},
     TITLE = {Central limit theorem in complete feedback games},
   JOURNAL = {J. Appl. Probab.},
  FJOURNAL = {Journal of Applied Probability},
    VOLUME = {61},
      YEAR = {2024},
    NUMBER = {2},
     PAGES = {654--666},
      ISSN = {0021-9002,1475-6072},
   MRCLASS = {60F05 (60G70 91A60)},
  MRNUMBER = {4740956},
MRREVIEWER = {Lina\ Mallozzi},
       DOI = {10.1017/jpr.2023.64},
       URL = {https://doi.org/10.1017/jpr.2023.64},
}

@misc{Tripathi2026,
      title={On the position matrix of single-shelf shuffle and card guessing}, 
      author={Tripathi, Raghavendra },
      year={2026},
      eprint={2602.07920},
      archivePrefix={arXiv},
      primaryClass={math.PR},
      url={https://arxiv.org/abs/2602.07920}, 
}

@misc{Tri26cut,
      title={Cutoff for asymmetric shelf shuffle}, 
      author={Tripathi, Raghavendra},
      year={2026},
      eprint={2606.18039},
      archivePrefix={arXiv},
      primaryClass={math.PR},
      url={https://arxiv.org/abs/2606.18039}, 
}

@misc{tri26analysis,
      title={Analysis of the asymmetric shelf shuffle}, 
      author={Tripathi, Raghavendra },
      year={2026},
      eprint={2606.18047},
      archivePrefix={arXiv},
      primaryClass={math.PR},
      url={https://arxiv.org/abs/2606.18047}, 
}

@book {FlaSed,
    AUTHOR = {Flajolet, Philippe and Sedgewick, Robert},
     TITLE = {Analytic combinatorics},
 PUBLISHER = {Cambridge University Press, Cambridge},
      YEAR = {2009},
     PAGES = {xiv+810},
      ISBN = {978-0-521-89806-5},
   MRCLASS = {05-02 (05A15 05A16 60C05 60E10 82-01)},
  MRNUMBER = {2483235},
       DOI = {10.1017/CBO9780511801655},
       URL = {https://doi.org/10.1017/CBO9780511801655},
}

@article {BayerDiaconis,
    AUTHOR = {Bayer, Dave and Diaconis, Persi},
     TITLE = {Trailing the dovetail shuffle to its lair},
   JOURNAL = {Ann. Appl. Probab.},
  FJOURNAL = {The Annals of Applied Probability},
    VOLUME = {2},
      YEAR = {1992},
    NUMBER = {2},
     PAGES = {294--313},
      ISSN = {1050-5164,2168-8737},
   MRCLASS = {60C05 (20B30 60B15)},
  MRNUMBER = {1161056},
MRREVIEWER = {David\ J.\ Aldous},
       URL = {http://links.jstor.org/sici?sici=1050-5164(199205)2:2<294:TTDSTI>2.0.CO;2-F&origin=MSN},
}

@book {DZ10,
    AUTHOR = {Dembo, Amir and Zeitouni, Ofer},
     TITLE = {Large deviations techniques and applications},
    SERIES = {Stochastic Modelling and Applied Probability},
    VOLUME = {38},
      NOTE = {Corrected reprint of the second (1998) edition},
 PUBLISHER = {Springer-Verlag, Berlin},
      YEAR = {2010},
     PAGES = {xvi+396},
      ISBN = {978-3-642-03310-0},
   MRCLASS = {60F10},
  MRNUMBER = {2571413},
       DOI = {10.1007/978-3-642-03311-7},
       URL = {https://doi.org/10.1007/978-3-642-03311-7},
}

@article {SW14,
    AUTHOR = {Saumard, Adrien and Wellner, Jon A.},
     TITLE = {Log-concavity and strong log-concavity: A review},
   JOURNAL = {Stat. Surv.},
  FJOURNAL = {Statistics Surveys},
    VOLUME = {8},
      YEAR = {2014},
     PAGES = {45--114},
      ISSN = {1935-7516},
   MRCLASS = {60E05 (60E15 62E10 62H05)},
  MRNUMBER = {3290441},
MRREVIEWER = {Roger\ B.\ Nelsen},
       DOI = {10.1214/14-SS107},
       URL = {https://doi.org/10.1214/14-SS107},
}

@article {Pitman,
    AUTHOR = {Pitman, Jim},
     TITLE = {Probabilistic bounds on the coefficients of polynomials with
              only real zeros},
   JOURNAL = {J. Combin. Theory Ser. A},
  FJOURNAL = {Journal of Combinatorial Theory. Series A},
    VOLUME = {77},
      YEAR = {1997},
    NUMBER = {2},
     PAGES = {279--303},
      ISSN = {0097-3165,1096-0899},
   MRCLASS = {05A16 (05A20 60C05 60G99)},
  MRNUMBER = {1429082},
MRREVIEWER = {Edward\ A.\ Bender},
       DOI = {10.1006/jcta.1997.2747},
       URL = {https://doi.org/10.1006/jcta.1997.2747},
}

@book {Varadhan,
    AUTHOR = {Varadhan, S. R. S.},
     TITLE = {Large deviations},
    SERIES = {Courant Lecture Notes in Mathematics},
    VOLUME = {27},
 PUBLISHER = {Courant Institute of Mathematical Sciences, New York; American Mathematical Society, Providence, RI},
      YEAR = {2016},
     PAGES = {vii+104},
      ISBN = {978-0-8218-4086-3},
   MRCLASS = {60F10 (60-02 60H10 60J25 60K35)},
  MRNUMBER = {3561097},
MRREVIEWER = {Max\ Fathi},
}

@article {LambertWFunction,
    AUTHOR = {Corless, R. M. and Gonnet, G. H. and Hare, D. E. G. and
              Jeffrey, D. J. and Knuth, D. E.},
     TITLE = {On the {L}ambert {$W$} function},
   JOURNAL = {Adv. Comput. Math.},
  FJOURNAL = {Advances in Computational Mathematics},
    VOLUME = {5},
      YEAR = {1996},
    NUMBER = {4},
     PAGES = {329--359},
      ISSN = {1019-7168,1572-9044},
   MRCLASS = {33E20},
  MRNUMBER = {1414285},
       DOI = {10.1007/BF02124750},
       URL = {https://doi.org/10.1007/BF02124750},
}

@article {GMTW,
    AUTHOR = {Gross, Jonathan L. and Mansour, Toufik and Tucker, Thomas W. and Wang, David G. L.},
     TITLE = {Log-concavity of combinations of sequences and applications to
              genus distributions},
   JOURNAL = {SIAM J. Discrete Math.},
  FJOURNAL = {SIAM Journal on Discrete Mathematics},
    VOLUME = {29},
      YEAR = {2015},
    NUMBER = {2},
     PAGES = {1002--1029},
      ISSN = {0895-4801,1095-7146},
   MRCLASS = {05C10 (05A15 05A20)},
  MRNUMBER = {3355766},
MRREVIEWER = {Yan\ Yang},
       DOI = {10.1137/140978867},
       URL = {https://doi.org/10.1137/140978867},
}

@article{StanleyLogConc,
author = {Stanley, R.},
journal = {Annals of the New York Academy of Sciences},
title = {Log-Concave and Unimodal Sequences in Algebra, Combinatorics, and Geometry},
year = {1989},
volume = {576},
pages = {500-535}}

@article {Ciucu98Nofeedback,
    AUTHOR = {Ciucu, Mihai},
     TITLE = {No-feedback card guessing for dovetail shuffles},
   JOURNAL = {Ann. Appl. Probab.},
  FJOURNAL = {The Annals of Applied Probability},
    VOLUME = {8},
      YEAR = {1998},
    NUMBER = {4},
     PAGES = {1251--1269},
      ISSN = {1050-5164,2168-8737},
   MRCLASS = {60C05 (60G40)},
  MRNUMBER = {1661184},
MRREVIEWER = {David\ J.\ Aldous},
       DOI = {10.1214/aoap/1028903379},
       URL = {https://doi.org/10.1214/aoap/1028903379},
}

\end{document}